\documentclass[11pt,reqno]{amsart}

\usepackage{amssymb,amsmath,amsthm,graphicx,xcolor,mathtools,mathrsfs,tabularx,bbm,tikz,url}
\usepackage{graphicx, float}

\usepackage{fullpage}
\usepackage{enumerate, enumitem}
\usepackage{bbm}
\usepackage{pgf,tikz}
\usetikzlibrary{arrows}
\usetikzlibrary{positioning}
\usetikzlibrary{patterns}

\def\COMMENT#1{}
\let\COMMENT=\footnote
\def\TASK#1{}

\def\COMMENT#1{}

\usepackage{comment}

\usepackage{setspace}

 \usepackage{geometry}
\newcommand{\eps}{\varepsilon} 
\newcommand{\N}{\mathbb{N}} 
\newcommand{\F}{\mathcal{F}} 

\usepackage[mathlines]{lineno}
\usepackage{etoolbox} 

\newcommand*\linenomathpatch[1]{%
	\expandafter\pretocmd\csname #1\endcsname {\linenomath}{}{}%
	\expandafter\pretocmd\csname #1*\endcsname{\linenomath}{}{}%
	\expandafter\apptocmd\csname end#1\endcsname {\endlinenomath}{}{}%
	\expandafter\apptocmd\csname end#1*\endcsname{\endlinenomath}{}{}%
}
\newcommand*\linenomathpatchAMS[1]{%
	\expandafter\pretocmd\csname #1\endcsname {\linenomathAMS}{}{}%
	\expandafter\pretocmd\csname #1*\endcsname{\linenomathAMS}{}{}%
	\expandafter\apptocmd\csname end#1\endcsname {\endlinenomath}{}{}%
	\expandafter\apptocmd\csname end#1*\endcsname{\endlinenomath}{}{}%
}

\expandafter\ifx\linenomath\linenomathWithnumbers
\let\linenomathAMS\linenomathWithnumbers
\patchcmd\linenomathAMS{\advance\postdisplaypenalty\linenopenalty}{}{}{}
\else
\let\linenomathAMS\linenomathNonumbers
\fi

\linenomathpatchAMS{gather}
\linenomathpatchAMS{multline}
\linenomathpatchAMS{align}
\linenomathpatchAMS{alignat}
\linenomathpatchAMS{flalign}
\linenomathpatch{equation}

\newenvironment{proofclaim}[1][Proof of the claim]{\begin{proof}[#1]}{\end{proof}}

\usepackage{cleveref}

\theoremstyle{plain}
\newtheorem{theorem}{Theorem}[section]
\crefname{theorem}{Theorem}{Theorems}

\newtheorem{proposition}[theorem]{Proposition}
\crefname{proposition}{Proposition}{Propositions}

\newtheorem{corollary}[theorem]{Corollary}
\crefname{corollary}{Corollary}{Corollaries}

\newtheorem{lemma}[theorem]{Lemma}
\crefname{lemma}{Lemma}{Lemmata}

\newtheorem{conjecture}[theorem]{Conjecture}
\crefname{conjecture}{Conjecture}{Conjectures}

\crefname{problem}{Problem}{Problem}

\newtheorem{claim}[theorem]{Claim}
\crefname{claim}{Claim}{Claims}

\newtheorem{observation}[theorem]{Observation}
\crefname{observation}{Observation}{Observations}

\crefname{setup}{Setup}{Setups}

\crefname{fact}{Fact}{Facts}

\crefname{algoritheorem}{Algoritheorem}{Algoritheorems}

\crefname{remark}{Remark}{Remarks}

\crefname{example}{Example}{Examples}

\theoremstyle{definition}
\newtheorem{definition}[theorem]{Definition}
\crefname{definition}{Definition}{Definitions}

\crefname{construction}{Construction}{Constructions}

\crefname{question}{Question}{Questions}

\numberwithin{equation}{section}

\crefname{section}{Section}{Sections}
\crefname{appendix}{Appendix}{Appendix}

\crefname{figure}{Figure}{Figures}

\def\COMMENT#1{}
\let\COMMENT=\footnote          

\usepackage{comment}

\let\polishlcross=\l
\def\l{\ifmmode\ell\else\polishlcross\fi}

\renewcommand{\rho}{\varrho}
\newcommand{\sm}{\setminus}
\renewcommand{\subset}{\subseteq}

\newcommand{\NATS}{\mathbb{N}}

\newcommand{\INTS}{\mathbb{Z}}

\let\vn\relax
\newcommand{\vn}{\mathbf{1}}

\newcommand{\cL}{\mathcal{L}}

\author[Hi\d{\^e}p H\`an]{Hi\d{\^e}p H\`an}
\address[Hi\d{\^e}p H\`an]{Departamento de Matem\'atica y Ciencia de la Computaci\'on, Universidad de Santiago de Chile, Las Sophoras 173, Estaci\'on Central, Santiago, Chile
}
\email{hiep.han@usach.cl}

\author[Richard Lang]{Richard Lang}
\address[Richard Lang]{Departament de Matemàtiques, Universitat Politècnica de Catalunya, Spain}
\email{richard.lang@upc.edu}

\author[Jo\~ao Pedro Marciano]{Jo\~ao Pedro Marciano}
\address[Jo\~ao Pedro Marciano]{IMPA, Estrada Dona Castorina 110, Jardim Bot\^anico, Rio de Janeiro, RJ, Brazil}
\email{joao.marciano@impa.br}

\author[Mat\'ias Pavez-Sign\'e]{Mat\'ias Pavez-Sign\'e}
\address[Mat\'ias Pavez-Sign\'e]{Centro de Modelamiento Matemático (CNRS IRL2807), Universidad de Chile, Chile}
\email{mpavez@dim.uchile.cl}

\author[Nicol\'as Sanhueza-Matamala]{Nicol\'as Sanhueza-Matamala}
\address[Nicol\'as Sanhueza-Matamala]{Departamento de Ingeniería Matemática, Facultad de Ciencias Físicas y Matemáticas, Universidad de Concepción, Chile}
\email{nsanhuezam@udec.cl}

\author[Andrew Treglown]{Andrew Treglown}
\address[Andrew Treglown]{University of Birmingham, United Kingdom}
\email{a.c.treglown@bham.ac.uk}
\author[Camila Z\'arate-Guer\'en]{Camila Z\'arate-Guer\'en}
\address[Camila Z\'arate-Guer\'en]{University of Birmingham, United Kingdom}	
\email{ciz230@student.bham.ac.uk}

\allowdisplaybreaks

\begin{document}
	\title{Colour-bias perfect matchings in hypergraphs}

    
	\begin{abstract}
		We study conditions under which an edge-coloured hypergraph has a particular substructure that contains more than the trivially guaranteed number of monochromatic edges.
		Our main result solves this problem for perfect matchings under minimum degree conditions.
		This answers recent questions of Gishboliner, Glock and Sgueglia, and of Balogh, Treglown and Z\'arate-Guer\'en.
	\end{abstract}
	\maketitle

    \vspace{-1.0cm}
	
	\section{Introduction}
	An old problem of Erd\H{o}s concerns the existence of colour-bias substructures in edge-coloured hypergraphs~\cite{Erd63,ES72}; that is, a copy of a particular hypergraph $F$ in an $r$-edge-coloured hypergraph $H$ that contains significantly more than a $1/r$-proportion of its edges in the same colour.
	We investigate this question in the case when $H$ has large minimum degree and $F$ is a perfect matching.
	
	Colour-bias problems for spanning structures were initially studied for graphs.
	A classic result of Dirac states that every graph $G$ on $n \geq  3$ vertices with minimum degree $\delta(G)\geq  n/2$ contains a Hamilton cycle; moreover, this minimum degree condition is sharp as simple constructions show.
	Balogh, Csaba, Jing and Pluh\'ar~\cite{BCJ+20} gave a colour-bias analogue of this result: given any $\eps>0$, there exists a $\gamma>0$ so that if $G$ is a sufficiently large $n$-vertex graph with  $\delta(G)\geq (3/4 + \eps) n$, then any $2$-edge-colouring of $G$ admits a Hamilton cycle with at least $(1/2 + \gamma)n$ monochromatic edges. Moreover, for $n$ divisible by $4$, there are $2$-edge-coloured $n$-vertex graphs $G$ with $\delta (G)=3n/4$ so that \emph{every} Hamilton cycle in $G$ contains precisely $n/2$ edges in each colour.
	Subsequently, this result was generalised to $r$-edge-colourings by Freschi, Hyde, Lada and Treglown~\cite{FHL+21} and Gishboliner, Krivelevich and Michaeli~\cite{GKP22}.\footnote{In these results, the $3/4$ in the minimum degree condition is replaced by $(r+1)/2r$.} 
	Analogous results (in the $2$-edge-coloured setting) have been established for $K_r$-factors by Balogh, Csaba,  Pluh\'ar and Treglown~\cite{BCP+21}, and more generally for $H$-factors by Brada\v{c}, Christoph and Gishboliner~\cite{BCG+23}.
	A version for $k$th powers of  Hamilton cycles  was proved by Brada\v{c}~\cite{Bra22}.
	Colour-bias problems have also been considered for random graphs by Gishboliner, Krivelevich and Michaeli~\cite{GKM22}.

	In what follows, we focus on related questions for hypergraphs.
	Formally, a \emph{$k$-uniform hypergraph} (\emph{$k$-graph} for short)
	$H$ has a set of \emph{vertices}~$V(H)$ and a set of \emph{edges}~$E(H)$, where each edge consists of $k$ vertices.
 For $1 \leq \ell \leq k-1$, the
	\emph{minimum $\ell$-degree of $H$}, denoted $\delta_\ell(H)$, is the maximum
	$m$ such that every set of $\ell$ vertices in $H$ is contained in at least $m$
	edges.
	The $\ell=1$ case is referred to as the \emph{minimum vertex degree}; the $\ell=k-1$ case is the  \emph{minimum codegree}.
	
	Colour-bias problems for hypergraphs have been investigated for tight Hamilton cycles and perfect matchings.
	Mansilla Brito~\cite{Man23} gave a minimum codegree result for forcing 
	a colour-bias copy of a tight Hamilton cycle in a $2$-edge-coloured $3$-graph. 
	More generally, Gishboliner,  Glock and  Sgueglia~\cite{GGS23} determined optimal minimum codegree conditions for colour-bias tight Hamilton cycles in $r$-edge-coloured $k$-graphs for all $r\geq 2$ and $k\geq 3$.
	
	 Recall that a 
	\emph{perfect matching} in a hypergraph $H$ is a collection of vertex-disjoint edges that covers all vertices of $H$.
	For $1\leq \ell <k$, we define~$m_{\ell,k}$ as the \emph{(asymptotic)
		minimum $\ell$-degree existence threshold for perfect matchings}. More precisely,
	$m_{\ell,k}$ is the infimum $c \in [0,1]$ such that for every $\eps
	>0$ and $n$ sufficiently large and divisible by $k$, every $n$-vertex $k$-graph $H$ with
	$\delta_\ell(H) \geq (c + \eps)\binom{n-\ell}{k-\ell}$ contains a perfect matching.
	
	A simple consequence of Dirac's theorem is that $m_{1,2} = 1/2$.
	This was extended to higher uniformities and codegrees by R\"odl, Ruci\'nski and Szemer\'edi~\cite{RRS09b}, who in fact proved a more exact result.
	There is a large body of work for lower degree types, and we refer the reader to the survey of Zhao~\cite{Zha16} for a more detailed history.
	We remark, however, that the minimum vertex degree case of the problem is largely open, and $m_{1,k}$ is thus far only known for $k \leq 5$~\cite{AFH+12, HPS09, Kha16}. 
	It is widely believed that the thresholds are attained by certain simple partite constructions (see, e.g.,~\cite{Zha16}), which leads to the central problem of the area:
	\begin{conjecture} \label{conjecture:matchingthresholds}
		For every $1 \leq \ell < k$, we have $m_{\ell,k} = \max 
		\left \{ 
		\frac{1}{2}, 1- \left ( \frac{k-1}{k}\right )^{k-\ell}  \right\}.$
	\end{conjecture}
	
	What can we say about colour-bias in perfect matchings? For codegrees, this was solved by the above mentioned work on tight Hamilton cycles~\cite[Corollary~1.2]{GGS23} as well as by Balogh, Treglown and Z\'arate-Guer\'en~\cite[Theorem~1.3]{BTZ24}.  
    In fact, this latter work  determined the asymptotically optimal minimum $\ell$-degree condition for forcing a colour-bias perfect matching in an $r$-edge-coloured $k$-graph on $n$ vertices whenever $2 \leq \ell < k$ and $r\geq 2$.
	Their result shows that a relative minimum $\ell$-degree of $m_{\ell,k} + \eps$ forces a perfect matching containing at least $n/(kr)+\gamma n$ monochromatic edges, where $\gamma>0$ is small with respect to $\eps$, $k$ and $r$.
	In other words, colour-bias perfect matchings are born at the same time as perfect matchings when $2 \leq \ell < k$.
	
	\subsection*{Results}
	
	We study this problem for \emph{all} choices of $\ell$.
	Our main result states that the minimum $\ell$-degree threshold for the existence of a colour-bias perfect matching for any $r$-edge-colouring is the maximum between the existence threshold $m_{\ell,k}$ and the maximum minimum $\ell$-degree of a certain family of $k$-graphs, which we will denote by $\F_{k,r}$.
	In particular, whilst we already mentioned that in 
 the range $2 \leq \ell < k$, 
 the existence threshold $m_{\ell,k}$ prevails,  for $\ell =1$ a more nuanced picture emerges.
	
	To describe our family of $k$-graphs $\F_{k,r}$, we need the following terminology.
	Let $H$ be a $k$-graph and $\{ V_1,V_2, \dotsc, V_r\}$ be an $r$-partition of $V(H)$.
	We say that an edge $e \in E(H)$ is of \emph{type} $\mathbf{j} = (j_1,j_2, \ldots, j_r)$ with respect to $(V_1, V_2, \ldots, V_r)$ if $|e \cap V_i|=j_i$ for all $i \in [r]$.
	Note that since every edge of $H$ has size $k$, it holds that $\sum_{i=1}^r j_i = k$ and $j_i \geq 0$ for each $i \in [r]$.
	We also let $\mathbf{e}_i \in \{0,1\}^r$ be the canonical vector with $1$ in the $i$th position and $0$ otherwise.
	We say that a pair $(\mathbf{j},\sigma)$, consisting of a vector $\mathbf j=(j_1, j_2, \dotsc, j_r) \in \mathbb{N}_0 ^r$ and $\sigma \in \{-1, 1\}$, is \emph{$k$-valid} if $\sigma + \sum_{i=1}^r j_i = k$, and $j_i + \sigma \geq 0$ for all $i \in [r]$.
	
	\begin{definition}\label{def:extremal-graphs}
		For $k, r \geq 2$, and given a $k$-valid pair $(\mathbf{j} ,\sigma)$, we say that an $r$-edge-coloured $k$-graph $H$ belongs to the family $\F^\ast_{k,r}(\mathbf{j}, \sigma)$ if there is a partition $\{V_1, \dotsc, V_r \}$ of $V(H)$ 
		such that, for each $i \in [r]$, the $i$-coloured edges of $H$ are   $k$-sets of type $\mathbf{j} + \sigma \mathbf{e}_i$
		with respect to $(V_1, \ldots, V_r)$. Note that this does not necessarily mean \emph{all} edges of type $\mathbf{j} + \sigma \mathbf{e}_i$ are present in $H$ though. If they are all present  for all choices of $i \in [r]$, we say that $H$ is \emph{edge-maximal}.
		
		We also let $\mathcal{F}_{k,r}(\mathbf{j}, \sigma) \subseteq \mathcal{F}^\ast_{k,r}(\mathbf{j}, \sigma)$ be the family of $k$-graphs $H$ where $n := |V(H)|$ is divisible by $kr$ and the corresponding vertex partition satisfies, for every $i \in [r]$, that $|V_i| = \frac{rj_i+\sigma}{rk}n$.

        We define $\F^\ast_{k,r} := \bigcup_{\mathbf{j}, \sigma} \F^\ast_{k,r} (\mathbf{j}, \sigma)$ and $\F_{k,r} := \bigcup_{\mathbf{j}, \sigma} \F_{k,r} (\mathbf{j}, \sigma)$, where both unions range over all  $k$-valid pairs $(\mathbf{j}, \sigma)$.
		
	\end{definition}
	
	The crucial property of this family is that for every $n$-vertex $k$-graph in $\F_{k,r}$, every perfect matching has exactly $n/(kr)$ edges in each colour (see \cref{lemma:colourproportion} for a proof, and \cref{figure:3uniform2colours} for an example).
	Thus, colour-bias in perfect matchings requires a minimum degree that exceeds the one found in the family $\F_{k,r}$.
	We therefore define \[f_{\ell,k,r}:= 
	\lim_{n \rightarrow \infty} \max_{\substack{H \in \F_{k,r} \\ |V(H)| = krn}} \frac{\delta_\ell(H)}{\binom{|V(H)| - \ell}{k - \ell}}.\]

	\begin{figure}[h]
		\centering
		\includegraphics[scale=0.9]{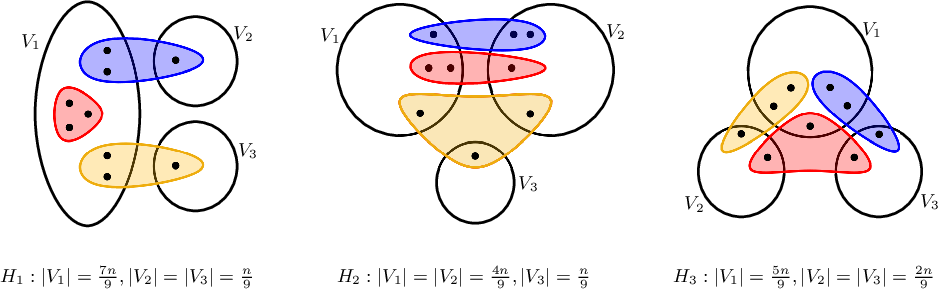}
		\caption{The edge-maximal  members of $\mathcal{F}_{k,r}$ for $k = r = 3$. Up to symmetry, there are three different cases, corresponding to different possible choices of $3$-valid pairs $(j_1, j_2, j_3) \in \N_0^3$, and $\sigma \in \{-1, +1\}$. The choices correspond to $(j_1, j_2, j_3, \sigma) \in \{ (2,0,0, +1), (1,1,0, +1), (2,1,1, -1)\}$ and are pictured as $H_1, H_2, H_3$, respectively.
			In each case, the red, blue and yellow edges represent the $1$-coloured, $2$-coloured, $3$-coloured edges, respectively.
			The minimum vertex degree of those three $3$-graphs are $(49/81 + o(1))\binom{n-1}{2}$, $(32/81 + o(1))\binom{n-1}{2}$ and $(5/9 + o(1))\binom{n-1}{2}$, respectively. The maximum $\delta_1(H_i)$ is attained by $\delta_1(H_1)$. This shows that $f_{1, 3, 3} = 49/81$.}
		\label{figure:3uniform2colours}
	\end{figure}
	
	Our main result states that every $r$-edge-coloured $k$-graph whose minimum $\ell$-degree forces a perfect matching and at the same time denies membership in $\F_{k,r}$ must already exhibit a colour-bias perfect matching.
	
	\begin{theorem}\label{thm:main}
	Let $\eps >0$ and $\ell,k, r \in \mathbb N$ where	 $1 \leq \ell < k$ and $r \geq 2$.  There exist $\gamma > 0$ and $n_0\in \mathbb N$ such that if $n \geq n_0$ is divisible by $k$, then every $r$-edge-coloured $k$-graph $H$ on $n$ vertices with 
		\[\delta_\ell(H) \geq (\max\{f_{\ell,k,r},\, m_{\ell,k}\} + \eps) \tbinom{n-\ell}{k-\ell}\]
		has a perfect matching containing at least ${n}/{(kr)}+ \gamma n$ edges of the same colour.
	\end{theorem}
Our result answers a question posed by
Gishboliner, Glock and Sgueglia. In~\cite[Section 8]{GGS23} they  asked for which values of $\ell, k, r$ the 
 minimum $\ell$-degree threshold for colour-bias
  perfect matchings in $r$-edge-coloured $k$-graphs  exceeds the corresponding existence threshold.
  Theorem~\ref{thm:main} tells us that this is the case precisely when 
  $f_{\ell,k,r}>m_{\ell,k}$.

As discussed earlier, Balogh,  Treglown and Z\'arate-Guer\'en~\cite{BTZ24} showed that, for $2\leq \ell <k$ and $r \geq 2$, the colour-bias perfect matching threshold coincides with the (ordinary) threshold for perfect matchings $m_{\ell,k}$.
  It is not hard to see that \Cref{thm:main} recovers this  result.

	\begin{observation} \label{obs:btz}
        For all $r \geq 2$ and $2 \leq \ell < k$, we have $f_{\ell, k, r} \leq m_{\ell, k}$.
    \end{observation}

    \begin{proof}
        Let $r, \ell, k$ be as in the statement.
        Since $m_{\ell, k} \geq m_{k-1, k} = \frac{1}{2}$, it suffices to show that $f_{\ell, k, r} \leq \frac{1}{2}$.
        Suppose otherwise.
        This implies  that there exists an $n \geq 2rk$, a $k$-valid $r$-tuple $(\mathbf{j}, \sigma)$, and an $n$-vertex $k$-graph $H \in \mathcal{F}_{k,r}(\mathbf{j}, \sigma)$ such that $\delta_2(H) > \frac{1}{2} \binom{n}{k-2}$.
        Since $H \in \mathcal{F}_{k,r}(\mathbf{j}, \sigma)$, there exists a  partition $\{V_1, \dotsc, V_r\}$ of $V(H)$ such that $|V_i| = (r j_i + \sigma)n/(rk) \geq n/rk \geq 2$ for all $i \in [r]$.
        Let $\{x_1, y_1\} \subseteq V_1$ and $\{x_2, y_2\} \subseteq V_2$.
        By the minimum $2$-degree condition, there exists a $(k-2)$-set $S$ such that both $e_1 = \{x_1, y_1\} \cup S$ and $e_2 \{x_2, y_2\} \cup S$ are edges in $H$.
        Then $|e_1 \cap V_1| = |e_2 \cap V_1| + 2$, $|e_1 \cap V_2| = |e_2 \cap V_2| - 2$, and $|e_1 \cap V_i| = |e_2 \cap V_i|$ for all $3 \leq i \leq r$.
        But then it is not true that all edges of $H$ are of type $\mathbf{j} + \sigma \mathbf{e}_i$ (for some $i$) with respect to $(V_1, \dotsc, V_r)$, a contradiction.
    \end{proof}
  
  However, as we will see, for $\ell=1$ the inequality $f_{1, k, r} > m_{\ell, k}$ can hold, depending on the values of $k$ and $r$.
  In light of \Cref{thm:main}, to determine the minimum vertex degree threshold for colour-bias perfect matchings, we need to understand the values of $f_{1,k,r},$ denoted by $f_{k,r}$ for short.
  
	
	Given $k$ and $r$, consider the edge-maximal $r$-edge-coloured $k$-graph in $\F_{k,r}$ generated by the $k$-valid pair $(\mathbf{j}, \sigma)$ where $\mathbf{j} = (k-1, 0, \dotsc, 0)$ and $\sigma = 1$.
	This $k$-graph consists of $r$ parts, the first $V_1$ of size $(kr - r + 1)n/kr$ and all of the others $V_2, \dotsc, V_r$ of size $n/kr$; the $i$-coloured edges contain $k-1$ vertices in $V_1$ and the remaining one in $V_i$, for each $1 \leq i \leq k$.
	(In \Cref{figure:3uniform2colours}, this $k$-graph is depicted as $H_1$ in the case $k =r= 3$.)
	In fact, this specific $k$-graph is the one that maximises the minimum vertex degree in the definition of $f_{k, r}$, except in a couple of cases.
	
	\begin{lemma} \label{lemma:maximiser}
        For each $k, r \geq 2$ such that $(k,r) \not\in \{ (3,2), (4,2) \}$, we have
		\[ f_{k,r} = \left( 1- \frac{r-1}{kr} \right)^{k-1}. \]
	\end{lemma}
    We obtain this formula in Section~\ref{sec:numer} as a consequence of \cref{lemma:fkr-reduction1,lemma:fkr-reduction2} and a simple computer search.
    The exceptional cases $k \in \{3, 4\}$ and  $r=2$ are resolved as follows.
    \begin{lemma} \label{lemma:maximiser-remaining}
        $f_{3,2} = 3/4$ and $f_{4,2} = 175/256$.
    \end{lemma}
    We remark that $f_{3,2} = 3/4$ is attained by the edge-maximal $2$-edge-coloured graph in $\mathcal{F}_{3,2}$ generated by the $3$-valid pair $(\mathbf{j}, \sigma)$ with $\mathbf{j} = (1,1)$ and $\sigma = 1$; and $f_{4,2} = 175/256$ is attained by the edge-maximal $2$-edge-coloured graph in $\mathcal{F}_{4,2}$ generated by the $4$-valid pair $(\mathbf{j}, \sigma)$ with $\mathbf{j} = (2,1)$ and $\sigma = 1$.

    Next, we state a few direct consequences of Theorem~\ref{thm:main} and the previous two lemmata, obtained via some simple calculations.
    Abbreviate the vertex-degree threshold for perfect matchings by $m_k=m_{1,k}$, and let $b_{k,r}=\max\{f_{k,r},m_k\}$ denote the minimum vertex degree threshold for colour-bias perfect matchings.

    For graphs, that is when $k=2$, and any $r \geq 3$, the inequality $f_{2, r} = (r+1)/(2r) > 1/2 = m_{2}$ holds (see \cref{lemma:maximizer-graph}).
    This implies that $b_{2,r} = (r+1)/(2r)$, and  recovers the known graph results, which follow from the work on Hamilton cycles~\cite{FHL+21, GKP22}.
    
    Using Lemma~\ref{lemma:maximiser-remaining} and the known values of $m_3$ and $m_4$, together with Theorem~\ref{thm:main} we have that $b_{3,2} = 3/4$ and $b_{4,2} = 175/256$.
	The case $(k,r) = (3,2)$ solves a question posed in \cite[Question 4.2]{BTZ24}.
%
	
    For general values of $k,r$, we can decide whether the ordinary and colour-biased thresholds coincide (formally $b_{k,r} = m_k$) by checking whether $f_{k,r} < m_k$.
	The next corollary summarises our findings for two colours.
	
\begin{corollary}\label{cor:vertex}
 If \Cref{conjecture:matchingthresholds} is true, then
 $b_{k,2}>m_k$ if and only if $2\leq k\leq16$. 
 Unconditionally, we have $b_{k,2} > m_k$ for
 $2 \leq k \leq 5$   and $b_{k,2} = m_{k}$ for $k \geq 17$. 
	\end{corollary}
	
	For $k, r \geq 3$ we  obtain the following complete answer.
	
	\begin{corollary} \label{corollary:fversusm}
		Let $k,r \geq 3$.
		Then $b_{k,r} = m_{k}$ if and only if $(k,r) \notin \{ (3,3), (4,3), (3,4) \}$.
	\end{corollary}

 \smallskip

 The paper is organised as follows. In the next section we prove a key lemma used in the proof of Theorem~\ref{thm:main}. In Section~\ref{sec:main-thm} we prove 
	Theorem~\ref{thm:main}. The calculations needed to deduce Lemma~\ref{lemma:maximiser} are provided in Section~\ref{sec:numer}. We finish the paper with some concluding remarks in Section~\ref{sec:conc}.
	
	\smallskip
 
{\bf Remark:} Just before submitting this paper,  we learnt of simultaneous and independent work of Lu, Ma and Xie~\cite[Theorem 4]{luma}
who have also proved a version of the $\ell=1$ case of Theorem~\ref{thm:main}.
	
	\section{Switchers from large common neighbourhoods}
	
	In this section, we introduce the notion of a switcher, which is a gadget that allows us to manipulate the colour profile of a matching.
	Roughly speaking, our key contribution, \cref{lem:main}, states that $k$-graphs with sufficiently large common neighbourhoods and no small switchers must be members of the extremal family $\F^\ast_{k,r}$ from \cref{def:extremal-graphs}.
	From this, \cref{thm:main} can be easily derived (see \Cref{sec:main-thm}).
	
	Formally, an \emph{$i$-switcher} is the union of two $r$-edge-coloured matchings $M_1$ and $M_2$ with $V(M_1) = V(M_2)$ such that $M_1$ has more edges of colour $i$ than $M_2$.
	Its \emph{order}  is $|V(M_1)|$. A \emph{switcher} is an $i$-switcher for some $i\in[r]$.
	For a vertex $x$ in a $k$-graph $H$, we write $N(x)$ (or sometimes $N_H(x)$ for clarity) for its \emph{neighbourhood}, which is the set of $(k-1)$-sets $Y$ such that $\{x\} \cup Y$ is an edge in $H$.
	
	\begin{lemma}[Key Lemma]\label{lem:main}
		Let $H$ be a $k$-graph on $n$ vertices such that for every $x,y \in V(H)$, there are at least $k^2$  vertex-disjoint $(k-1)$-sets in $N(x) \cap N(y)$.
		Suppose $H$ has an $r$-edge-colouring with at least one edge of each colour and no switcher of order at most $k^2+k$. Then $H \in \F^\ast_{k,r}$.
	\end{lemma}
	
	For the proof of \cref{lem:main} we require some additional vocabulary.
	Given an $r$-edge-coloured $k$-graph $H$ and $x \in V(H)$, we write $N_i(x)$ for the $(k-1)$-sets  $Y \in N(x)$ such that $\{x\} \cup Y$ forms an edge of colour $i$.
	A directed graph $D$ consist of a set of \emph{vertices} $V(D)$ and a set of \emph{(directed) edges} $E(D)$, where each  edge is a pair $(x,y)$ of vertices.
	For a vertex $x \in V(D)$, we denote by $N^+_D(x)$ its \emph{out-neighbourhood}, which is the set of all vertices $y$ such that $(x,y) \in E(D)$.
	The \emph{in-neighbourhood} $N^-_D{(x)}$ is defined analogously.
	
	\begin{proof}[Proof of \cref{lem:main}]
		Let $c\colon E(H)\to [r]$ be an $r$-colouring such that there is at least one edge of each colour and $H$ is free of switchers of order at most $k^2+k$. 
		
		\begin{claim}\label{claim:switcher:1}
			Let $x,y \in V(H)$ be distinct, and let $i,j\in [r]$ be distinct colours.
			Suppose that there is some $S \in N(x) \cap N(y)$ such that 
			\[c(\{x\} \cup S)=i\quad \text{and}\quad c(\{y\} \cup S)=j.\]
			Then, for every $S' \in N(x) \cap N(y)$, we have that $c(\{x\} \cup S')=i$ and $c(\{y\} \cup S')=j$.
			In other words, $N(x) \cap N(y) = N_i(x) \cap N_j(y)$.
		\end{claim}
		
		\begin{proofclaim}
			Let $S' \in N(x) \cap N(y)$, and assume first that $S'$ is disjoint from~$S$.
			If $c(\{x\} \cup S') \neq i$ or $c(\{y\} \cup S') \neq j$, then define $M_1 := \{ S\cup \{x\}, S'\cup\{y\} \}$ and $M_2: = \{ S\cup\{y\}, S'\cup \{x\} \}$.
			So $M_1 \cup M_2$ forms a switcher of order $2k \leq k^2 + k$, a contradiction.
			
			Now assume that $S$ and $S'$ intersect.
			By assumption, there exists $S'' \in N(x) \cap N(y)$ disjoint both from $S$ and $S'$, and we can repeat the argument above, first with $S$ and $S''$ and then with $S''$ and $S'$.
			Thus the claim holds for all $S' \in N(x) \cap N(y)$, as desired.
		\end{proofclaim}

		For each colour $i \in [r]$, let  $D_i$ be the directed graph with vertex set $V(D_i)=V(H)$ such that $(x,y) \in E(D_i)$ if there exists at least one $S\in N(x)\cap N(y)$ such that $c(\{x\} \cup S)=i$ and $c(\{y\}\cup S)\not =i$.
		Recall that by assumption, there are at least $k^2$  vertex-disjoint $S\in N(x)\cap N(y)$.		
		So if $(x,y) \in E(D_i)$, then \cref{claim:switcher:1} implies that there are at least $k^2$  vertex-disjoint $S\in N(x)\cap N(y)$ with $c(\{x\} \cup S)=i$ and $c(\{y\}\cup S)\not =i$.

		%
		
		\begin{claim}\label{claim:nonempty}
			The edge set $E(D_i)$ is non-empty for every colour $i \in [r]$.
		\end{claim}
		\begin{proofclaim}
			For the sake of a contradiction, suppose that $E(D_i)=\emptyset$ for some colour~$i$.
			By an assumption in the lemma, there is an edge $e$ of colour $i$ and an edge $f$ of another colour.
			Let  $\{x_1, \dotsc, x_s\} := e \setminus f$ and $\{ y_1, \dotsc, y_s\} := f \setminus e$.
			We can greedily select pairwise disjoint $(k-1)$-sets $S_1, \dotsc, S_s$, such that for each $t\in[s]$, $S_t$ is disjoint from $e\cup f$ and $S_t \in N(x_t) \cap N(y_t)$.
			Observe that for $t \in [s]$, it follows that $\{x_t\} \cup S_t$ and $\{y_t\} \cup S_t$ must either be both of colour $i$ or both have colour distinct from $i$, since $D_i$  neither contains $(x_t,y_t)$ nor $(y_t,x_t)$.
			Define
			\begin{align*}
				M_1  : = \{ e \} \cup \{  \{y_t\} \cup S_t \colon t \in [s] \} \quad \text{and}  \quad
				M_2   := \{ f \} \cup \{  \{x_t\} \cup S_t \colon t \in [s] \}.
			\end{align*}
			Then $M_1 \cup M_2$ is an $i$-switcher of order $k-s + s(k+1) \leq k^2 + k$, a contradiction.
		\end{proofclaim}
		
		Note that \cref{claim:switcher:1} easily implies that $D_i$ has no directed cycle of length two.
		We now show that $D_i$ cannot have directed paths of length two.
		
		\begin{claim}\label{claim:noP2}
			For each $i\in [r]$, $D_i$ contains no directed path with two edges.
		\end{claim}
		
		\begin{proofclaim}
			For the sake of a contradiction, suppose that both $(x,y)$ and $(y,z)$ are edges in $D_i$.
			Thus there are colours $j,\ell$, both distinct from $i$, such that 
			we may pick $S\in N_i(x)\cap N_j(y)$, $T\in N_i(y)\cap N_\ell(z)$ and $U\in N(x)\cap N(z)$, all pairwise disjoint.
			Let
			\begin{align*}
				M_1   := \{ S\cup\{x\}, T\cup \{y\}, U\cup \{z\} \} \quad \text{and}  \quad
				M_2   := \{ S\cup \{y\}, T\cup \{z\}, 
				U\cup \{x\} \}.
			\end{align*}
   Then  $M_1 \cup M_2$
			is an $i$-switcher of order $3k \leq k^2+k$, a contradiction.
		\end{proofclaim}
		
		For each colour $i \in [r]$, define
		\begin{equation*}
			V_i^+:=\big\{x \in V(D_i) \colon |N^+_{D_i}(x) |>0\big\} \text{ and } V_i^-:=\big\{x \in V(D_i) \colon |N^-_{D_i}(x)|>0\big\}.
		\end{equation*}
		Note that \cref{claim:noP2} readily implies that $V_i^+$ and $V_i^-$ are disjoint.
		We now prove that they partition the vertex set $V(D_i)=V(H)$ and that the underlying graph is complete bipartite. 
		
		\begin{claim}\label{claim:bipcomplete}
			For every $i \in [r]$, we have that $V_i^+ \cup V_i^- = V(H)$ and $E(D_i) = V_i^+ \times V_i^-$.
		\end{claim}
		
		\begin{proofclaim}
			First, we observe that for an arbitrary $(x,y)\in E(D_i)$ and any vertex $z \neq x,y$, one of the four pairs $(x,z), (z,x), (y,z), (z,y)$ must belong to $E(D_i)$.
			Suppose otherwise.
			Since $(x,y) \in E(D_i)$, there is a colour $j\not= i$ such that there are at least~$k^2$  disjoint sets in $N_i(x)\cap N_j(y)$.
			Pick pairwise disjoint sets $S\in N_i(x)\cap N_j(y)$, $T\in N(x)\cap N(z)$ and $U\in N(y)\cap N(z)$.
			The assumption $(x,z), (z,x) \notin E(D_i)$ implies that either both $c(\{x\} \cup T)$ and $c(\{z\} \cup T)$ are equal to $i$, or both are distinct from $i$. 
			Similarly, either both $c(\{y\} \cup U)$ and $c(\{z\} \cup U)$ are equal to $i$ or both are distinct from $i$.
   Set
			\begin{align*}
				M_1   := \{S\cup\{x\}, T\cup \{z\}, U\cup \{y\}\}  \quad \text{and}  \quad
				M_2   := \{ S\cup \{y\}, T\cup \{x\} , U\cup \{z\} \}.
			\end{align*}
			Then  $M_1 \cup M_2$ is an $i$-switcher of order $3k \leq k^2 + k$, a contradiction.
			
			Now suppose that there is a vertex $z\in V(D_i)\setminus (V_i^+\cup V_i^-)$.
			Let $(x,y) \in E(D_i)$, and note that $z \neq x,y$.
			But then one of the four pairs $(x,z), (z,x), (y,z), (z,y)$ must belong to $E(D_i)$ by the above observation, which shows that $z \in V_i^+\cup V_i^-$, a contradiction.
			Since $V(D_i) = V(H)$, this confirms the first part of the claim.
			
			Next, we show that the edges of $D_i$ are precisely the pairs pointing from $V_i^+$ to $V_i^-$.
			Firstly, notice that there are no edges within $V_i^+$, within $V_i^-$, or going from $V^-_i$ to $V^+_i$, as otherwise we have a directed path of length two in $D_i$, thus contradicting \cref{claim:noP2}.
			Secondly, suppose that there are vertices $y\in V_i^+$ and 
			$z\in V_i^-$ such that $(y,z)\not\in E(D_i)$.
			As $z\in V_i^-$, there is an edge $(x,z)$ with $x\in V_i^+$ distinct from $y$.
			Since both $x$ and $y$ are in $V_i^+$, 
   we have $(x, y), (y, x)\not\in E(D_i)$.
			Moreover, $(z,y)$ is not an edge in $D_i$, since there is no path of length two in $D_i$.
			In summary, $(x,z)$ is an edge of $D_i$, while none of $(x,y)$, $(y,x)$, $(y,z)$ and $(z,y)$ are.
			But this contradicts our initial observation.
		\end{proofclaim}

		\begin{claim} \label{claim:partition}
			Either $\{V_i^+\}_{i \in [r]}$ or $\{V_i^-\}_{i \in [r]}$ forms a partition of  $V(H)$.
		\end{claim}
		
		\begin{proofclaim}
			We prove first that 
   \begin{align}\label{useful2}
       \text{for all distinct $i,j \in [r]$, we have  $V_i^+ \cap V_j^+ = \emptyset$ or $V_i^- \cap V_j^- = \emptyset$.}
   \end{align}
			For a contradiction, suppose that there are distinct $i,j\in [r]$ so that we can pick vertices $x\in V_i^+\cap V_j^+$ and $y\in V_i^-\cap V_j^-$. 
			By \cref{claim:bipcomplete}, we know that $(x,y)$ is an edge in both $D_i$ and $D_j$, and therefore there are colours $\ell,m\in [r]$, with $\ell\not =i$ and $m\not =j$, such that $|N_i(x)\cap N_\ell (y)|, |N_j(x)\cap N_m(y)|\geq 1$.
			This contradicts \cref{claim:switcher:1}; so  (\ref{useful2}) holds.

\smallskip 

   Next we show that $\{ V^-_i\}_{i \in [r]}$ or $\{ V^+_i\}_{i \in [r]}$ forms a collection of pairwise disjoint sets. If $r=2$, this follows from (\ref{useful2}); so suppose that $r \geq 3$.
Note that 
\begin{align}\label{useful}
    \text{for every distinct } i,j \in [r], \text{ we have } V^+_i \cap V^+_j \neq \emptyset  \text { or } V^-_i \cap V^-_j \neq \emptyset.
\end{align}
Indeed, otherwise 
$V^+_i = V^-_j$ and $V^-_i = V^+_j$ for some choice of $i,j$. Consider another colour $\ell \in [r]\setminus \{i,j\}$. By (\ref{useful2}), without loss of generality
we may assume that $V^+_i \cap V^+_\ell =\emptyset$; thus, 
$ V^+_\ell  \subseteq  V^-_i = V^+_j$. In particular, $V^+_\ell \cap V^+_j \neq \emptyset$ and \eqref{useful2} imply that $V^-_\ell \cap  V^-_j=\emptyset$. But then $V^-_\ell \subseteq  V^+_j$. So  $V^+_\ell\cup V^-_\ell=V(H)$ is contained in $V^+_j$, a contradiction to the fact that $V^-_j\subseteq V(H)$ is non-empty.
Hence, (\ref{useful}) holds.

			Pick distinct colours $i,j,\ell\in [r]$ and suppose that $V_i^+\cap V_j^+$, $V_j^+\cap V_\ell^+$
			and $V_\ell^-\cap V_i^-$ are non-empty. We then have $V_i^-\cap V_j^-=V_j^-\cap V_\ell^-=V_\ell^+\cap V_i^+=\emptyset$.
			This, however, implies that
			\[
			V^+_i\cap V_j^-= (V_i^+\cap  V_\ell^-\cap V_j^-) \cup
			(V^+_i\cap V_\ell^+\cap V_j^-)=\emptyset,
			\] 
			which together with $V_i^-\cap V_j^-=\emptyset$ implies that $V_j^-=\emptyset$,  a contradiction. 
			In conclusion, for any three distinct colours $i, j, \ell \in [r]$, one of the sets $V_i^+\cap V_j^+$, $V_j^+\cap V_\ell^+$ and $V_\ell^-\cap V_i^-$ must be empty.		
			An analogous argument shows that one of the sets $V^+_i \cap V^+_j$, $V^-_j \cap V^-_\ell$ and $V^-_\ell \cap V^-_i$ has to be empty also.
			
Consider a complete graph $K$ where $V(K)=[r]$.
Colour an edge $ij\in E(K)$ with $+$ if $V^+_i \cap V^+_j \neq \emptyset$; colour $ij\in E(K)$ with $-$ if $V^-_i \cap V^-_j \neq \emptyset$. By (\ref{useful2}) and (\ref{useful}), every edge in $K$ gets precisely one colour. Moreover, the argument in the previous paragraph shows that every triangle in $K$ is monochromatic (i.e., all $+$ or all $-$). This in turn  implies that $K$ itself is monochromatic.
Without loss of generality suppose that every edge in $K$ is coloured $-$. Then (\ref{useful2}) implies that 
for all distinct $i,j \in [r]$, we have  $V_i^+ \cap V_j^+ = \emptyset$. That is,  $\{ V^+_i\}_{i \in [r]}$ forms a collection of pairwise disjoint sets, as required.
   

   \smallskip 
			Finally, we argue that such a collection must also form a partition of $V(H)$.
			Indeed, suppose that $\{ V^+_i\}_{i \in [r]}$ are pairwise disjoint (the corresponding case with $\{ V^-_i\}_{i \in [r]}$ follows analogously).
			Fix any colour  $i \in [r]$ and recall that $V^+_i \cup  V^-_i = V(H)$ by \cref{claim:bipcomplete}.
			Moreover, $(x,y) \in E(D_i)$ for any $x \in V^+_i$ and $y \in V^-_i$.
			However, this implies that $(y,x) \in E(D_j)$ for some other colour $j \in [r]$ by definition of the directed graphs, and hence $y \in V_j^+$.
		\end{proofclaim}
		
		\begin{claim} \label{claim:typesvscolours}
			For every colour $i \in [r]$, there is some integer $0 \leq j \leq k-1$ such that every edge in $H$ has either  type $(j,k-j)$ or $(j+1,k-j-1)$ with respect to $(V_i^+,V_i^-)$.
			Moreover, $e\in E(H)$ has colour $i$ if and only if $e$ is of type $(j+1,k-j-1)$ with respect to $(V_i^+,V_i^-)$.
		\end{claim}
		
		\begin{proofclaim} Let $i\in[r]$ be given.
			Suppose that there are indices $0\le j<\ell\le k$ and edges $e,f \in E(H)$ such that $e$ has type $(j,k-j)$ and $f$ has type $(\ell,k-\ell)$, both with respect to $(V_i^+,V_i^-)$.
			We write $U := (e \cup f) \sm ( e \cap f)$, and note that $U \cap V_i^+$ and $U \cap V_i^-$ have each size at least $\ell - j$.
			Let us call a pair of vertices $(x,y) \in U^2$ \emph{crossing} if $x \in V_i^+$ and $y \in V_i^-$.
			By the above observation, we may write $U = \{x_1,y_1,\dots,x_q,y_q\}$ such that the number of crossing pairs $(x_t,y_t)$, denoted by $s$, satisfies $s \geq \ell - j$.
			
			Now, take sets $S_t \in N(x_t) \cap N(y_t)$, one for each $t \in [q]$.
			We can greedily pick $S_1, \dotsc, S_t$ to be pairwise disjoint and also disjoint from $e \cup f$, since by assumption we have at least $k^2$ vertex-disjoint choices in every step.
			Note that, crucially, if $(x_t,y_t)$ is crossing, then $c(\{ x_t\} \cup S_t) = i$ and $c(\{ y_t\} \cup S_t) \neq i$.
			Moreover, if $(x_t,y_t)$ is not crossing then $(x_t, y_t) \notin E(D_i)$ and thus $\{ x_t\} \cup S_t$ and $\{ y_t\} \cup S_t$ are either both of colour $i$ or both of a colour distinct from~$i$.
			
			Next, consider $G := M_1 \cup M_2 \subseteq E(H)$ where
			\begin{align*}
				M_1  := \{ e \} \cup \{ \{ y_t\} \cup S_t \colon t \in [q] \} \quad \text{and}  \quad
				M_2  := \{ f \} \cup \{ \{ x_t\} \cup S_t \colon t \in [q] \}.
			\end{align*}
			We remark that in any case, $G$ has $(k - q) + q(k+1) \leq k^2 +  k$ vertices.
			
			Moreover, note that $M_2$ has $s+1, s$ or $s-1$ more $i$-coloured edges than $M_1$, where the three possibilities occur if $c(f) = i \neq c(e)$; $c(e)$ and $ c(f)$ are both or neither equal to $i$; $c(e) = i \neq c(f)$, respectively.
			Thus, if $\ell \geq j+2$, then $s \geq 2$, and hence $M_2$ has always at least $s-1 \geq 1$ more $i$-coloured edge than $M_1$, and therefore $G$ is an $i$-switcher, which is a contradiction.
			Hence, we can only have $\ell=j+1$ and $s = 1$.
			In this case, the only way for $G$ to not be an $i$-switcher is that $c(e) = i$ and $c(f) \neq i$, which proves the second part of the statement.
		\end{proofclaim} 
		
		We now have all the tools to argue that $H \in \F^\ast_{k,r}$.
		By \cref{claim:typesvscolours}, for each $i \in [r]$ there exists $0 \leq f_i < k$ such that the $i$-coloured edges are precisely those of type $(f_i + 1, k - f_i - 1)$ with respect to $(V^+_i, V^-_i)$.
		By \cref{claim:partition}, one of $\{ V^+_i\}_{i \in [r]}$ or $\{ V^-_i\}_{i \in [r]}$ partitions $V(H)$.
		Assume first that $\{ V^+_i\}_{i \in [r]}$ partitions $V(H)$.
		We need to show that there exists $\mathbf{j} = (j_1, \ldots, j_r) \in \N^r_0$ and $\sigma \in \{-1, +1\}$ such that $\sigma + \sum_{i=1}^r j_i = k$, and all $i$-coloured edges are of type $\mathbf{j} + \sigma \mathbf{e}_i$ with respect to $(V^+_1, \dotsc, V^+_r)$.
		
		Let $\sigma := 1$ and $j_i := f_i$ for all $i \in [r]$.
		We claim that this choice satisfies the required properties.
		Indeed, let $e$ be a $1$-coloured edge.
		By our choice of $f_1$, $e$ must be of type $(f_1+1, k - f_1-1)$ with respect to $(V^+_1, V^-_1)$.
		Now, let $x_2, \dotsc, x_r$ be such that $e$ is of type $(f_1 + 1, x_2, \dotsc, x_r)$ with respect to $(V^+_1,V^+_2, \dotsc, V^+_r)$.
		For any $i \neq 1$, since $e$ is not $i$-coloured it must be of type $(f_i, k - f_i)$ with respect to $(V^+_i, V^-_i)$.
		This implies that $f_i = x_i$.
		Thus, we have that \[ \sigma + \sum_{i=1}^r j_i = 1 + \sum_{i=1}^r f_i = (1 + f_1) + \sum_{i=2}^r x_i = k,\] as required.
		Also, our choice implies that the $i$-coloured edges are precisely those that are of type $(j_i + 1, k - j_i - 1)$ with respect to $(V^+_i, V^-_i)$, which is equivalent to being  of type $\mathbf{j} + \sigma \mathbf{e}_i$ with respect to $(V^+_1, \dotsc, V^+_r)$.
		
		Finally, in the case that $\{ V^-_i\}_{i \in [r]}$ partitions $V(H)$, we instead set $\sigma := -1$, and for each $i \in [r]$ we set $j_i := k - f_i$.
		With these changes, the proof is analogous to the previous case.
	\end{proof}

	\section{Proof of the Main Theorem}\label{sec:main-thm}
	
	This section is dedicated to the proof of \cref{thm:main}.
	We require the following two technical lemmata.
	
	\begin{lemma} \label{lemma:colourproportion}
 Let $k,r \geq 2$ and let $n \in \mathbb N$ be divisible by $k$.
		Let $H \in \mathcal{F}^\ast_{k,r}$ be an $n$-vertex $r$-edge-coloured $k$-graph with corresponding partition $\{V_1, \dotsc, V_r\}$, $r$-tuple $(j_1, \dotsc, j_r)$ and $\sigma \in \{-1, +1\}$.
		Suppose $\alpha_1, \dotsc, \alpha_r \geq 0$ are such that $\sum_{i=1}^r \alpha_i = 1$, and for every $i \in [r]$, we have $|V_i| = (j_i + \sigma \alpha_i )n/k$.
		Then every perfect matching in $H$ contains exactly $\alpha_i n /k$ $i$-coloured edges for every $i \in [r]$.
	\end{lemma}
	
	\begin{proof}
		Let $M$ be a perfect matching in $H$.
		For each $i \in [r]$, let $x_i$ be the number of $i$-coloured edges in $M$.
		Note that $\sum_{i=1}^r x_i = n/k$.
		Let $i \in [r]$ be arbitrary.
		Note that each non-$i$-coloured edge of $H$ intersects $V_i$ precisely in $j_i$ vertices, and every $i$-coloured edge of $H$ intersects $V_i$ precisely in $j_i + \sigma$ vertices.
		Therefore,
		\[ (j_i + \sigma \alpha_i)\frac{n}{k} = |V_i| = j_i \left(\frac{n}{k} - x_i \right) + (j_i + \sigma)x_i = j_i \frac{n}{k} + x_i \sigma, \]
		which solves to $x_i = \alpha_i n/k$, as desired.
	\end{proof}
	
	\begin{lemma}\label{lem:continuity}
	Let $k,r \geq 2$.	For every $\beta > 0$ and $1\leq \ell < k$, there exist $\gamma > 0$ and $n_0\in \mathbb N$ such that the following holds.
		Let $H \in \mathcal{F}^\ast_{k,r}$ be an $r$-edge-coloured $k$-graph on $n\geq n_0$ vertices with a perfect matching which has at most $n/(rk) + \gamma n$ edges of each colour.
		Then $\delta_\ell(H) \leq (f_{\ell,k,r} + \beta) \binom{n-\ell}{k-\ell}$.
	\end{lemma}
	
	\begin{proof}
		Given $\beta >0$ and $1\leq \ell < k$, we choose $\gamma>0$ sufficiently small and then $n_0\in \mathbb N$ sufficiently large with respect to $\gamma$.
		Let $H$ be a $k$-graph as in the statement of the lemma.
		By definition of $\mathcal{F}_{k,r}^\ast$, there is a partition $\{V_1,\ldots, V_r\}$ of $V(H)$, a tuple $\mathbf{j}=(j_1,\ldots, j_r)\in\mathbb N_0^r$, and $\sigma\in\{-1,+1\}$ such that, for each $i\in [r]$, the $i$-coloured edges of $H$ are those of type $\mathbf{j}+\sigma\mathbf{e}_i$ with respect to $(V_1,\ldots, V_r)$.
		
		Let $M$ be a perfect matching in $H$ as in the statement of the lemma and, for $i\in [r]$, let $M_i$ be the set of $i$-coloured edges of $M$. 
		By assumption, we have that $|M_i| \leq \left( 1/r + \gamma k\right)n/k$ for each $i \in [r]$.
		This also implies that $|M_i| \geq \left( 1/r - \gamma 
 r k\right)n/k$, and thus $|M_i| = \left(1/r \pm \gamma r k\right)n/k$.
Furthermore, since $H$ contains a perfect matching, if $\sigma =1$ we know that 
$j_i n/k \leq |V_i| \leq (j_i+1)n/k$ for every $i \in [r]$;
if $\sigma =-1$ we know that 
$(j_i-1) n/k \leq |V_i| \leq j_in/k$ for every $i \in [r]$.
In particular, this implies there are 
$\alpha_1, \dotsc, \alpha_r \geq 0$  such that $\sum_{i=1}^r \alpha_i = 1$, and for every $i \in [r]$, we have $|V_i| = (j_i + \sigma \alpha_i )n/k$.
 
		Our observations in the last paragraph allow us to apply \cref{lemma:colourproportion} to conclude that, for every $i \in [r]$, we have
		\begin{align} \label{equation:Vi_size} |V_i| = \left( j_i + \sigma \left(\frac{1}{r} \pm \gamma rk \right) \right)\frac{n}{k}.
		\end{align}

		Let $X \subset V(H)$ be an arbitrary set of $\ell$ vertices.
		For each $i \in [r]$, we write $p_i$ for the number of vertices that $X$ has in $V_i$.
		Let $d_s(X)$ denote the number of $s$-coloured edges in $H$ containing $X$.
		It follows that
		\[d_s(X)\le \binom{|V_s|}{j_s+\sigma - p_s}  \prod_{i\in [r]\setminus \{s\}} \binom{|V_i|}{j_i-p_i}.\]
		Therefore, using~\eqref{equation:Vi_size}, the definition of $f_{\ell,k,r}$ and the choice of $\gamma$ and $n_0 \leq n$, we have
		\[d(X)=\sum_{i\in [r]}d_i(X)\le (f_{\ell,k,r}+\beta)\binom{n-\ell}{k-\ell},\]
		as required.
	\end{proof}

	Now we are ready to derive our main result.
	
	\begin{proof}[Proof of \cref{thm:main}]
Let $\eps >0$ and $\ell,k, r \in \mathbb N$ where	 $1 \leq \ell < k$ and $r \geq 2$. 
Set $\beta := \eps/3$ and let $0<\gamma '\leq \beta /(k^3+k^2)$ be such that the statement of 
\cref{lem:continuity} holds with $\gamma'$ playing the role of $\gamma$. Then choose
		 $0<\gamma \leq \gamma '/(2r^2)$ and let
		 $n_0\in \mathbb N$ be sufficiently large.
  
   Let $n \geq n_0$ be divisible by $k$.
    Let $H$ be an $n$-vertex  $k$-graph with $\delta_\ell(H)\ge (\max\{f_{\ell,k,r}, m_{\ell, k}\}+\varepsilon)\tbinom{n-\ell}{k-\ell}$; by definition of $m_{\ell, k}$, $H$ contains a perfect matching.
		For the sake of a contradiction,  suppose that there is an $r$-edge-colouring of $H$ such that for every perfect matching $M$ in $H$ and every colour $i\in [r]$, the number of $i$-coloured edges in $M$ is at most $n/(kr)+\gamma n$.
		
		Let $\mathcal{S}$ be a maximal collection of vertex-disjoint switchers of order at most $k^2 + k$ in~$H$.
		
		\begin{claim}
			The size of $\mathcal{S}$ is at most $\gamma ' n$.
		\end{claim}
		
		\begin{proofclaim}
			Assume otherwise.
			Recall that a switcher is the union of two matchings $M_1$ and $M_2$  with $V(M_1) = V(M_2)$ together with a colour $i$, such that $M_1$ has more edges of colour $i$ than $M_2$.
			We call $M_1$ and $M_2$ its \emph{majority} and \emph{minority state}, respectively.			
			Clearly there exists some colour $i \in [r]$, such that 
   there is a collection $\mathcal{S}_i \subseteq \mathcal{S}$ of  
   $i$-switchers where $|\mathcal{S}_i|= \lceil \gamma ' n /r \rceil$.
   
			Let $H'\subseteq H$ be the $k$-graph obtained from $H$ by removing $V(\mathcal{S}_i)$. The choice of $\gamma '$ ensures that
			 $\delta _{\ell} (H')\geq (\max\{f_{\ell,k,r},\, m_{\ell,k}\} + \eps/2) \tbinom{n-\ell}{k-\ell}$. Thus, $H'$ contains a  perfect matching $M$.
			
			Let $M'$ be the union of $M$ with the minority state matchings of $\mathcal{S}_i$; so $M'$ is a perfect matching in $H$.
			By assumption, $M'$ has at most $n/(kr)+\gamma n$ edges of each colour.
			So in particular, $M'$ has at least $n/(kr) - (r-1)\gamma n$ edges of colour $i$.
			Let $M''$ be the union of $M$ with the majority state matchings of $\mathcal{S}_i$.
			It follows that $M''$ has at least $\gamma ' n/r > r\gamma n$ more edges of colour $i$ than $M'$, and thus more than $n/(kr) + \gamma n$ edges of colour $i$ in total,
			a contradiction.
		\end{proofclaim}
		
		Let $H'\subseteq H$ be the $k$-graph obtained from $H$ by removing $V(\mathcal{S})$.
		Thus,
        \begin{equation}
            \delta_\ell(H') \geq (\max\{f_{\ell,k,r},\, m_{\ell,k}\} + \eps/2) \tbinom{n-\ell}{k-\ell}. \label{equation:minldegreeH'}
        \end{equation} 
		So there is a perfect matching $M$ in $H'$.
		By assumption, $M$ has at most $n/(kr) +  \gamma n\leq |V(H')|/(kr)+\gamma '|V(H')|$ edges of each colour. Indeed otherwise, we could choose one matching from each switcher in $\mathcal{S}$ and add it to $M$, resulting in a perfect matching of $H$ containing at least $n/(kr) +  \gamma n$ edges of one colour, a contradiction.
		
		Next, we show that $H' \in \mathcal{F}^\ast_{k,r}$ by applying \cref{lem:main}.
		To this end, note that $H'$ does not contain any switcher of order at most $k^2 + k$  by the maximality of $\mathcal{S}$.
		Moreover, since $M$ has at most $n/(kr) +  \gamma n$ edges of each colour, it has at least one edge of each colour.
        Therefore, $H'$ still has at least one edge of each colour.
		Finally, we observe that every two vertices $x,y \in V(H')$ have at least $k^2$ vertex-disjoint $(k-1)$-sets in their common neighbourhood in $H'$.
		Indeed, as $m_{\ell,k} \geq 1/2$ (see, e.g.,~\cite{Zha16})
		we infer from \eqref{equation:minldegreeH'}  that $\delta_1 (H')\geq (1/2 + \eps/2) \tbinom{n-1}{k-1}$.
		This shows that $|N_{H'}(x) \cap N_{H'}(y) |\geq (\eps/2) \tbinom{n-1}{k-1}$, and thus we may greedily select $k^2$ vertex-disjoint elements.
		We may therefore apply \cref{lem:main} to deduce that $H' \in \mathcal{F}^\ast_{k,r}$.
  Recall that $M$ is a perfect matching in $H'$  with at most $|V(H')|/(kr)+\gamma '|V(H')|$ edges of each colour. Therefore, we may apply
 \cref{lem:continuity} to deduce that $\delta_\ell(H') \leq (f_{\ell,k,r} + \eps/3) \binom{n-\ell}{k-\ell}$, which contradicts \eqref{equation:minldegreeH'}.
	\end{proof}
	
	\section{Numerology}\label{sec:numer}
	
	In this section we are interested in quantitative aspects of $f_{ k, r}$.
    Our main goal is to prove \cref{lemma:maximiser}, which determines $f_{k,r}$ in all but two cases.

	Our first goal is to give a formula for $f_{k,r}$ that is simpler to calculate.
	Given a $k$-valid pair $(\mathbf{j}, \sigma)$ with $\mathbf{j} \in \NATS_0^r$ and $\sigma \in \{-1,1\}$, we define
	\[ f_{(\mathbf{j}, \sigma)}: = 
	\lim_{n \rightarrow \infty}\max_{\substack{H \in \F_{k,r}(\mathbf{j}, \sigma) \\ |V(H)| = krn}} \frac{\delta_1(H)}{\binom{|V(H)| - 1}{k - 1}}, \]
	and thus we have
	\[ f_{k,r} = \max_{\substack{ \mathbf{j} \in \mathbb{N}^r_0, \, \sigma \in \{-1, 1\} \\ ( \mathbf{j}, \sigma) \text{ is $k$-valid}}} f_{(\mathbf{j}, \sigma)}.\]
	Therefore, our first task will be to understand $f_{(\mathbf{j}, \sigma)}$ for a given $k$-valid pair $(\mathbf{j}, \sigma)$. The next lemma gives a formula for $f_{(\mathbf{j}, \sigma)}$ which removes the limit and the dependency on $n$.
    We denote the \emph{multinomial coefficient} of non-negative integers $n,k_1,\dots,k_m$ by $\binom{n}{k_1,\dots,k_m} := \frac{n!}{k_1! \dots k_m!}$.
 
	\begin{lemma} \label{lemma:fkr-reduction1}
		For each $k, r \geq 2$ and each $k$-valid $(\mathbf{j}, \sigma)$, we have
		\[ f_{(\mathbf{j}, \sigma)} =  \min_{1 \leq i \leq r} \left\lbrace \frac{1}{k^k r^{k-1}(r j_i + \sigma)} \sum_{\ell = 1}^r j_{\ell, i} \binom{k}{j_{\ell, 1}, \dotsc, j_{\ell, r}} \prod_{s=1}^r (r j_s + \sigma )^{j_{\ell, s}} \right\rbrace ,\]
		where $j_{\ell, s}$ is the $s$th coordinate of the vector $\mathbf{j}+\sigma \mathbf{e}_\ell$; that is, $j_{\ell,s} = j_s$ if $\ell \neq s$, and $j_s + \sigma$ if $\ell = s$.
	\end{lemma}
	
	\begin{proof}
		Recall that, given $n$ divisible by $kr$, the
        size of the classes $V_1, \dotsc, V_r$ in any $n$-vertex $k$-graph in the family $\F_{k, r}(\mathbf{j}, \sigma)$ are fixed and equal to $|V_i| = (r j_i + \sigma)n/rk$ for each $1 \leq i \leq k$. Amongst such 
        $k$-graphs $H$, the maximum $\delta_1(H)$ is attained when $H$ is a $k$-graph in $\F_{k, r}(\mathbf{j}, \sigma)$ where all the possible edges of a given type are present; from now on let $H$ be such a $k$-graph on $n$ vertices.
		
		To understand $\delta_1(H)$, first we will calculate the number of $\ell$-coloured edges incident to a vertex $v_i \in V_i$, for each $i, \ell \in [r]$.
		To do this, note that the $\ell$-coloured edges of $H$ are of type $\mathbf{j}+\sigma \mathbf{e}_\ell$.
		Suppose first that $j_{\ell, i} \neq 0$.
		Then the $\ell$-coloured degree $d_\ell(v_i)$ of $v_i$ is equal to
		\[ d_\ell(v_i) = \binom{|V_i| - 1}{j_{\ell, i} - 1} \left[ \prod_{s \neq i} \binom{|V_s|}{j_{\ell, s}} \right]. \]
		Since we are interested in the limiting value as $n \rightarrow \infty$, we can use the approximation $\binom{\alpha n}{k} = (1+o(1))\alpha^k n^k / k!$ to write the formula above as
		\begin{align*}
			d_\ell(v_i)
			& = (1+o(1)) \left( \frac{r j_i + \sigma}{rk}\right)^{j_{\ell, i} - 1} \frac{n^{j_{\ell, i} - 1}}{(j_{\ell, i} - 1)!} \left[ \prod_{s \neq i} \left( \frac{r j_s + \sigma}{rk} \right)^{j_{\ell, s}} \frac{n^{j_{\ell, s}}}{j_{\ell, s}!} \right] \\
			& = (1+o(1)) \left[ \prod_{s=1}^r \left( \frac{r j_s + \sigma}{rk} \right)^{j_{\ell, s}} \frac{n^{j_{\ell, s}}}{j_{\ell, s}!} \right] \frac{rk}{r j_i + \sigma} \frac{j_{\ell, i}}{n} \\
			& = (1+o(1)) \left[ \prod_{s=1}^r \frac{(r j_s + \sigma )^{j_{\ell, s}}}{j_{\ell, s}!} \right] \frac{j_{\ell, i}}{r j_i + \sigma} \left( \frac{n}{rk} \right)^{-1 + \sum_{s=1}^r j_{\ell, s}} \\
			& = (1+o(1)) \left[ \prod_{s=1}^r \frac{(r j_s + \sigma )^{j_{\ell, s}}}{j_{\ell, s}!} \right] \frac{j_{\ell, i}}{r j_i + \sigma} \left( \frac{n}{rk} \right)^{k-1} \\
			& = (1+o(1)) \left[ \prod_{s=1}^r \frac{(r j_s + \sigma )^{j_{\ell, s}}}{j_{\ell, s}!} \right] \frac{j_{\ell, i} (k-1)!}{r^{k-1}k^{k-1}(r j_i + \sigma)}  \frac{n^{k-1}}{(k-1)!} \\
			& = (1+o(1)) \left[ \prod_{s=1}^r (r j_s + \sigma )^{j_{\ell, s}} \right] \binom{k}{j_{\ell, 1}, \dotsc, j_{\ell, r}}\frac{j_{\ell, i} }{k^k r^{k-1}(r j_i + \sigma)} \binom{n-1}{k-1}.
		\end{align*}
		Now, note that if $j_{\ell, i} = 0$, then the $i$th coordinate of $\mathbf{j}+\sigma \mathbf{e}_\ell$ is zero, which implies that $v_i$ does not belong to any $\ell$-coloured edge and thus $d_\ell(v_i) = 0$.
		But the formula above coincides also in this case (because of the presence of the $j_{\ell, i}$), so it is valid in all cases.
		
		The degree $d_H(v_i)$ is obtained by considering the sum of $d_\ell(v_i)$ over all possible $\ell \in [r]$, and therefore
		\begin{align*}
			\frac{d_H(v_i)}{\binom{n-1}{k-1} } & = \frac{1+o(1) }{k^k r^{k-1}(r j_i + \sigma)} \sum_{\ell = 1}^r j_{\ell, i} \binom{k}{j_{\ell, 1}, \dotsc, j_{\ell, r}} \prod_{s=1}^r (r j_s + \sigma )^{j_{\ell, s}}.
		\end{align*}
		The minimum vertex degree of $H$ can be calculated by taking the minimum over $ i \in [r]$.
		This gives the claimed formula.
	\end{proof}
	
	The formula above is enough to compute the values of $f_{k, r}$ for small values of $k, r$ using a simple computer program; see \Cref{table:numerology}.
	
	\begin{table}[h]
		\centering
		\begin{tabular}{c|c||c|c|c|c|c|c|c|c|c}
			$k$ & $m'_{k}$ & $f_{k,2}$ & $f_{k,3}$ & $f_{k,4}$ & $f_{k,5}$ & $f_{k,6}$ & $f_{k,7}$ & $f_{k,8}$ & $f_{k,9}$ & $f_{k,10}$   \\ \hline
			3 & $0.5555 $ & $\color{red}{0.75}$ & $\color{red}{0.6049} $ & $\color{red}{0.5625}$ & $0.5377 $ & $0.5216 $ & $0.5102 $ & $0.5017 $ & $0.4951 $ & $0.4899 $ \\
			4 & $0.5781 $ & $\color{red}{0.6835} $ & $\color{red}{0.5787} $ & $0.5363 $ & $0.5120 $ & $0.4961 $ & $0.4850 $ & $0.4768 $ & $0.4705 $ & $0.4654 $\\
			5 & $0.5903 $ & $\color{red}{0.6561} $ & $0.5641 $ & $0.5220 $ & $0.4978 $ & $0.4822 $ & $0.4713 $ & $0.4632 $ & $0.4570 $ & $0.4521 $ \\
			6 & $0.5981 $ & $\color{red}{0.6472} $ & $0.5549 $ & $0.5129 $ & $0.4889 $ & $0.4734 $ & $0.4626 $ & $0.4546 $ & $0.4485 $ & $0.4437 $ \\
			7 & $0.6034 $ & $\color{red}{0.6410} $ & $0.5485 $ & $0.5066 $ & $0.4827 $ & $0.4674 $ & $0.4567 $ & $0.4487 $ & $0.4427 $ & $0.4379 $ \\
			8 & $0.6073 $ & $\color{red}{0.6365} $ & $0.5438 $ & $0.5020 $ & $0.4782 $ & $0.4630 $ & $0.4523 $ & $0.4444 $ & $0.4384 $ & $0.4336 $ \\
			9 & $0.6102 $ & $\color{red}{0.6330} $ & $0.5402 $ & $0.4985 $ & $0.4748 $ & $0.4596 $ & $0.4490 $ & $0.4412 $ & $0.4352 $ & $0.4304 $ \\
			10 & $0.6125 $ & $\color{red}{0.6302} $ & $0.5374 $ & $0.4957 $ & $0.4721 $ & $0.4569 $ & $0.4464 $ & $0.4386 $ & $0.4326 $ & $0.4279 $\\
			11 & $0.6144 $ & $\color{red}{0.6280} $ & $0.5351 $ & $0.4935 $ & $0.4699 $ & $0.4548 $ & $0.4442 $ & $0.4365 $ & $0.4305 $ & $0.4258 $\\
			12 & $0.6160 $ & $\color{red}{0.6261} $ & $0.5332 $ & $0.4916 $ & $0.4681 $ & $0.4530 $ & $0.4425 $ & $0.4348 $ & $0.4288 $ & $0.4241 $\\
			13 & $0.6173 $ & $\color{red}{0.6245} $ & $0.5316 $ & $0.4901 $ & $0.4666 $ & $0.4515 $ & $0.4410 $ & $0.4333 $ & $0.4274 $ & $0.4227 $\\
			14 & $0.6184 $ & $\color{red}{0.6232} $ & $0.5303 $ & $0.4888 $ & $0.4653 $ & $0.4503 $ & $0.4398 $ & $0.4321 $ & $0.4262 $ & $0.4215 $\\
			15 & $0.6193 $ & $\color{red}{0.6221} $ & $0.5291 $ & $0.4876 $ & $0.4642 $ & $0.4492 $ & $0.4387 $ & $0.4310 $ & $0.4251 $ & $0.4205 $\\
			16 & $0.6201 $ & $\color{red}{0.6211} $ & $0.5281 $ & $0.4866 $ & $0.4632 $ & $0.4482 $ & $0.4378 $ & $0.4301 $ & $0.4242 $ & $0.4196 $\\
			17 & $0.6209 $ & $0.6202 $ & $0.5272 $ & $0.4858 $ & $0.4624 $ & $0.4474 $ & $0.4370 $ & $0.4293 $ & $0.4234 $ & $0.4188 $\\
			18 & $0.6215 $ & $0.6194 $ & $0.5264 $ & $0.4850 $ & $0.4616 $ & $0.4467 $ & $0.4362 $ & $0.4286 $ & $0.4227 $ & $0.4181 $\\
			19 & $0.6221 $ & $0.6187 $ & $0.5257 $ & $0.4843 $ & $0.4610 $ & $0.4460 $ & $0.4356 $ & $0.4279 $ & $0.4221 $ & $0.4174 $\\
			20 & $0.6226 $ & $0.6181 $ & $0.5251 $ & $0.4837 $ & $0.4604 $ & $0.4454 $ & $0.4350 $ & $0.4274 $ & $0.4215 $ & $0.4169 $\\
            21 & $0.6231$ & $0.6175$ & $0.5245$ & $0.4831$ & $0.4598$ & $0.4449$ & $0.4345$ & $0.4269$ & $0.4210$ & $0.4164$ \\
            22 & $0.6235$ & $0.6170$ & $0.5240$ & $0.4826$ & $0.4593$ & $0.4444$ & $0.4340$ & $0.4264$ & $0.4205$ & $0.4159$
		\end{tabular}
		\vspace{1em}
		\caption{Computed values of $f_{k,r}$ for $3\leq k\leq 22$ and $2\leq r\leq 10$, truncated to the fourth decimal. Marked in red are the the values of $f_{k,r}$ that exceed $m_{k}':=1-\left(\frac{k-1}{k}\right)^{k-1}$, the conjectured value for $m_k$.}
		\label{table:numerology}
	\end{table}
	
	
	We further simplify the formula of \cref{lemma:fkr-reduction1} by removing the minimum.
	
	\begin{lemma} \label{lemma:fkr-reduction2}
		For each $k, r \geq 2$ let $(\mathbf{j}, \sigma)$ be a $k$-valid pair with $j_1 \geq j_2 \geq \dotsb \geq j_r$.
		Then we have
		\[ f_{(\mathbf{j}, \sigma)} =  \frac{\Pi(\mathbf{j}, \sigma) }{k^k r^{k-1}(r j_r + \sigma)} \left[ j_r \Lambda(\mathbf{j}, \sigma) + \sigma \left( \frac{(2k+1-\sigma)(r j_r + \sigma)}{2 j_r + 1 + \sigma} \right)^{\sigma} \right],  \]
		where
		\[ \Pi(\mathbf{j}, \sigma) \coloneqq \binom{k - \sigma}{j_{1}, \dotsc, j_{r}} \prod_{s=1}^r \left( r j_s + \sigma \right)^{j_{s}} \]
		and
		\[ \Lambda(\mathbf{j}, \sigma) \coloneqq \sum_{\ell = 1}^r \left( \frac{(2k+1-\sigma)(r j_\ell + \sigma)}{2 j_\ell + 1 + \sigma} \right)^{\sigma}. \]
	\end{lemma}
	
	\begin{proof}
		We start from the formula given by \cref{lemma:fkr-reduction1}.
		Let $d_i$ denote the $i$th term in the minimum of that formula, so that $f_{(\mathbf{j}, \sigma)} = \min\{ d_1, \dotsc, d_r\}$.
		The idea is to rewrite each $d_i$ by incorporating the multinomial coefficient with a $k-\sigma$ term as above.
		Using that $j_{\ell, i} = j_i$ if $\ell \neq i$ and $j_{i, i} = j_i + \sigma$, and distinguishing the cases $\sigma \in \{-1, +1\}$ we get
		\begin{align*}
			\binom{k}{j_{\ell, 1}, \dotsc, j_{\ell, r}} \prod_{s=1}^r (r j_s + \sigma )^{j_{\ell, s}}
			& = \left( \frac{(2k+1-\sigma)(r j_\ell + \sigma)}{2 j_\ell + 1 + \sigma} \right)^{\sigma} \binom{k - \sigma}{j_{1}, \dotsc, j_{r}} \prod_{s=1}^r (r j_s + \sigma )^{j_{s}} \\
			& = \left( \frac{(2k+1-\sigma)(r j_\ell + \sigma)}{2 j_\ell + 1 + \sigma} \right)^{\sigma} \Pi(\mathbf{j}, \sigma).
		\end{align*}
		So, using the values $j_{\ell, i}$ again, we get\begin{align*}
			d_i & = \frac{\Pi(\mathbf{j}, \sigma)}{k^k r^{k-1}(r j_i + \sigma)} \left[ j_i \Lambda(\mathbf{j}, \sigma) + \sigma \left( \frac{(2k+1-\sigma)(r j_i + \sigma)}{2 j_i + 1 + \sigma} \right)^{\sigma} \right]. 
		\end{align*}
		
		Note that the lemma states that $f_{(\mathbf{j}, \sigma)}=d_r = \min_{1 \leq i \leq r} d_i$.
		Let
		\[ g(\mathbf{j}, \sigma, i) := \frac{1}{r j_i + \sigma} \left[ j_i \Lambda(\mathbf{j}, \sigma) + \sigma \left( \frac{(2k+1-\sigma)(r j_i + \sigma)}{2 j_i + 1 + \sigma} \right)^{\sigma} \right], \]
		and note that it is enough to check that $i=r$ minimises $g(\mathbf{j}, \sigma, i)$.
		
		We begin our analysis with the case $\sigma = 1$.
		In this case, we have that
		\[ \Lambda(\mathbf{j}, 1) \coloneqq k \sum_{\ell = 1}^r \frac{r j_\ell + 1}{j_\ell + 1}. \]
		Note that the function $j \mapsto (r j +1)/(j+1)$ is monotone increasing, so each term in the sum is at least $\frac{r j_r + 1}{j_r + 1}$; therefore $\Lambda(\mathbf{j}, 1) \geq k r \frac{r j_r + 1}{j_r + 1} \geq kr$.
		On the other hand, we also have that each term in the sum is strictly less than $r$, so we also have $\Lambda(\mathbf{j}, 1) < kr^2$.
		
		Now, note that
		\[ g(\mathbf{j}, 1, i) = \frac{j_i \Lambda(\mathbf{j}, 1)}{r j_i + 1} + \frac{k}{j_i + 1}. \]
		We claim that this is minimised when $i=r$.
		To see this, consider the function
		\[ g(x) := \frac{\Lambda(\mathbf{j}, 1)x}{rx+1} + \frac{k}{x+1} \]
		for $x \geq 0$.
		We claim that $g(j_r) \leq g(x)$ holds for each $x \geq j_r$: this clearly implies the inequality we wanted.
		
		A straightforward calculation reveals that the derivative $g'$ is
		\[ g'(x) = \frac{\Lambda(\mathbf{j}, 1)(x+1)^2 - k(rx+1)^2}{(x+1)^2(rx+1)^2}. \]
		Since the denominator is positive for all $x \geq 0$, the change of sign of $g'$ is determined by the numerator, which is a quadratic function.
		The leading coefficient of the $x^2$ term in the numerator is $\Lambda(\mathbf{j},1) - kr^2 < 0$, so $g'(x)$ is eventually negative for large $x$.
		Moreover, we have that $g'(0) \geq 0$, so we can deduce  for positive $x$ that $g'(x)$ changes sign exactly once.
		Let  $x^\ast > 0$ be the  unique maximum of  $g(x)$; then $g(x)$ is monotone increasing in $[0, x^\ast)$ and monotone decreasing in $[x^\ast, \infty)$.
		To see that $g(j_r) \leq g(x)$ holds for each $x \geq j_r$ it is enough to check that $j_r \leq x^\ast$, and that the inequality $g(j_r) \leq \lim_{x \rightarrow \infty}g(x)$ holds.
		
		Indeed, we have
		\begin{align*}
			g(j_r)
			& = \frac{\Lambda(\mathbf{j}, 1) j_r}{r j_r + 1} + \frac{k}{j_r + 1}
			= \frac{\Lambda(\mathbf{j}, 1)}{r}  - \frac{\Lambda(\mathbf{j}, 1)}{r( r j_r + 1)} + \frac{k}{j_r + 1} 
			\leq \frac{\Lambda(\mathbf{j}, 1)}{r}  = \lim_{x \rightarrow \infty}g(x),
		\end{align*}
		where we used $\Lambda(\mathbf{j}, 1) \geq k r \frac{r j_r + 1}{j_r + 1}$ in the inequality.
		
		To see that $j_r \leq x^\ast$ it is enough to check that $g'(j_r) \geq 0$, for which it is enough to see that $\Lambda(\mathbf{j}, 1)(j_r+1)^2 - k(rj_r+1)^2 \geq 0$.
		This indeed holds, because
		\[ \left( \frac{rj_r + 1}{j_r + 1} \right)^2 < r \frac{rj_r + 1}{j_r + 1} \leq \frac{\Lambda(\mathbf{j}, 1)}{k}, \]
		where we used $\Lambda(\mathbf{j}, 1) \geq k r \frac{r j_r + 1}{j_r + 1}$ again.
		We conclude that indeed, if $\sigma = 1$, $d_i$ is minimised when $r=i$, as desired.
		
		The case where $\sigma = -1$ can be proven in an analogous way.
  \end{proof}

	Now, from the formula given by \cref{lemma:fkr-reduction2}, we can show that $f_{(\mathbf{j}, \sigma)}$ is maximised whenever $\sigma = 1$ and $\mathbf{j} = (k-1, 0, \dotsc, 0)$.
    We begin our analysis with the case $k=2$.

    \begin{lemma} \label{lemma:maximizer-graph}
        For all $r \geq 2$, $f_{2, r} = f_{((1,0,\dotsc, 0),1)} = \frac{r+1}{2r}$.
    \end{lemma}

    \begin{proof}
        For $\sigma = 1$, up to symmetry, the only possible $\mathbf{j}$ which makes $(\mathbf{j}, \sigma)$ a $2$-valid tuple is $(1, 0, \dotsc, 0)$; so it suffices to rule out the case $\sigma = -1$.
        For $\sigma = -1$, up to symmetry, the only possible $\mathbf{j}$ which make $(\mathbf{j}, \sigma)$ a $2$-valid tuple are $(2,1)$, for $r=2$, and $(1,1,1)$ for $r=3$.
        In the first case, we actually have $f_{((2,1),-1)} = f_{((1,0),1)}$; and in the second case we have $f_{((1,1,1),-1)} = \frac{2}{3} = f_{((1,0,0),1)}$.
    \end{proof}
    
    Next, we verify the fact for all large $k$ or $r$.

    \begin{lemma} \label{lemma:maximiser-weak}
		Let $k \geq 23$ or $r \geq 5$.
		Then $f_{k,r} = \left( 1- \frac{r-1}{kr} \right)^{k-1}$.
	\end{lemma}

	In fact, this is all we need to prove \cref{lemma:maximiser} and \cref{lemma:maximiser-remaining}, as the values not covered by the previous two lemmata are finitely many and can be checked by computer (all of them appear in \Cref{table:numerology}).
	We will need the following bounds, which follows by using sharp versions of Stirling's approximation to bound the factorials~\cite{Robbins1955}.
	
	\begin{lemma} \label{lemma:robbins}
        Let $r \geq 2$, $\sigma \in \{-1, +1\}$ and $j_1, \dotsc, j_r \geq 1$ such that $\sum_{i = 1}^r j_i = k - \sigma$. Then
		\[ \binom{k-\sigma}{j_1, \dotsc, j_r} \leq
		\frac{1}{\sqrt{2 \pi}^{r-1}}\sqrt{\frac{k-\sigma}{\prod_{i = 1}^r j_i}}\frac{(k-\sigma)^{(k-\sigma)}}{\prod_{i = 1}^r j_i^{j_i}}.\]
	\end{lemma}
	
	\begin{proof}[Proof of \cref{lemma:maximiser-weak}]
		Fix $k, r$ as in the statement, and let $f^\ast_{k,r} := \left( 1- \frac{r-1}{kr} \right)^{k-1}$.
		Note that in the case where $\sigma = 1$ and $\mathbf{j} = (k-1, 0, \dotsc, 0)$, we have that $f_{(\mathbf{j}, \sigma)} = f^\ast_{k,r}$.
		Hence, we need to show that for any other $k$-valid pair $(\mathbf{j}, \sigma)$, we have that $f_{(\mathbf{j}, \sigma)} \leq f^\ast_{k,r}$.
		
		We begin by considering the case $\sigma = 1$.
		Let $\mathbf{j} = (j_1, \dotsc, j_r)$ satisfy $j_1 \geq \dotsb \geq j_r \geq 0$ and $\sum_{\ell = 1}^r j_\ell = k-1$.
		By \cref{lemma:fkr-reduction2}, we have
		\begin{align*} f_{(\mathbf{j}, \sigma)}
			& =  \frac{ \prod_{\ell=1}^r \left( r j_\ell + 1 \right)^{j_{\ell}} }{k^k r^{k-1}(r j_r + 1)} \binom{k - 1}{j_{1}, \dotsc, j_{r}}  \left[ j_r \left( \sum_{\ell = 1}^r 
			\frac{k(rj_\ell + 1)}{j_\ell + 1} \right)  +  \frac{k(r j_r + 1)}{j_r + 1} \right] \\
			& = \frac{ \prod_{\ell=1}^r \left( r j_\ell + 1 \right)^{j_{\ell}} }{k^{k-1} r^{k-1}} \binom{k - 1}{j_{1}, \dotsc, j_{r}}  \left[ \frac{j_r}{r j_r + 1} \left( \sum_{\ell = 1}^r 
			\frac{rj_\ell + 1}{j_\ell + 1} \right)  +  \frac{1}{j_r + 1} \right] \\
			& = \frac{ \prod_{\ell=1}^r \left( r j_\ell + 1 \right)^{j_{\ell}} }{(kr)^{k-1}} \binom{k - 1}{j_{1}, \dotsc, j_{r}} W(\mathbf{j}),
		\end{align*}
		where in the last step we wrote \[W(\mathbf{j}) := \frac{j_r}{r j_r + 1} \left( \sum_{\ell = 1}^r 
		\frac{rj_\ell + 1}{j_\ell + 1} \right)  +  \frac{1}{j_r + 1}.\]
  
        Let $1 \leq s \leq r$ be the maximal index so that $j_{s} > 0$, which implies that $\prod_{\ell=1}^r \left( r j_\ell + 1 \right)^{j_{\ell}} = \prod_{\ell=1}^s \left( r j_\ell + 1 \right)^{j_{\ell}}$ and
        $\binom{k - 1}{j_{1}, \dotsc, j_{r}} = \binom{k - 1}{j_{1}, \dotsc, j_{s}}$.
        Using this, and writing $f^\ast_{k,r} = (r(k-1) + 1)^{k-1} / (kr)^{k-1}$, we have
		\begin{align} \frac{f_{(\mathbf{j}, \sigma)}}{f^\ast_{k,r}}
			& = \frac{ \prod_{\ell=1}^s \left( r j_\ell + 1 \right)^{j_{\ell}} }{(r (k-1) + 1)^{k-1}} \binom{k - 1}{j_{1}, \dotsc, j_{s}} W(\mathbf{j}). \label{equation:middlestep}
		\end{align}

		Since $j_1 \geq \dotsb \geq j_s \geq 1$, we can use \cref{lemma:robbins} to get
		\begin{align}
            \nonumber
			\frac{f_{(\mathbf{j}, \sigma)}}{f^\ast_{k,r}}
            & \leq \frac{1}{\sqrt{2 \pi}^{s-1}} \frac{ \prod_{\ell=1}^s \left( r j_\ell + 1 \right)^{j_{\ell}} }{(r (k-1) + 1)^{k-1}}\sqrt{\frac{k-1}{\prod_{\ell = 1}^s j_i}}\frac{(k-1)^{(k-1)}}{\prod_{\ell = 1}^s j_i^{j_i}} W(\mathbf{j}), \\ \nonumber
			& = \frac{1}{\sqrt{2 \pi}^{s-1}} \frac{ \prod_{\ell=1}^s \left( r + \frac{1}{j_\ell} \right)^{j_{\ell}} }{(r + \frac{1}{k-1})^{k-1}}\sqrt{\frac{k-1}{\prod_{\ell = 1}^s j_i}}W(\mathbf{j})\\ \nonumber
			& = \frac{1}{\sqrt{2 \pi}^{s-1}} \frac{ \prod_{\ell=1}^s \left( 1 + \frac{1}{r j_\ell} \right)^{j_{\ell}} }{(1 + \frac{1}{r(k-1)})^{k-1}}\sqrt{\frac{k-1}{\prod_{\ell = 1}^s j_i}}W(\mathbf{j}) \\
			& \leq \frac{r}{r+1}\frac{e^{s/r}}{\sqrt{2 \pi}^{s-1}} \sqrt{\frac{k-1}{\prod_{\ell = 1}^s j_i}}W(\mathbf{j}), \label{equation:middlestep2}
		\end{align}
		where in the last step we used that $(1 + 1/\left(rj_\ell)\right)^{j_\ell} \leq e^{1/r}$ in each term of the product, and we also used that $\left(1 + \frac{1}{r(k-1)}\right)^{k-1} \geq (r+1)/r$.
        Moreover, as $\sum_{\ell = 1}^s j_i = k-1$,
        \begin{equation}
            \prod_{\ell = 1}^s j_i \geq j_1 \geq \frac{k-1}{s},
            \label{equation:prodji1}
        \end{equation}
        and, since $j_2 \geq \dotsb \geq j_s \geq 1$, we also bound the product as follows: for a fixed value of $j_2$, the product of $j_i$'s is minimised when $j_3 = \dotsb = j_s = 1$ and $j_1=k-1 - j_2 - (s-2)$.
        Hence,
        \begin{equation}
            \prod_{\ell = 1}^s j_i \geq j_2(k - j_2 - s + 1).
            \label{equation:prodji2}
        \end{equation}
        Now we separate the analysis in some cases.
        \medskip

        \noindent \emph{Case 1: $s < r$.}
        In this case $j_r = 0$, and therefore $W(\mathbf{j}) = 1$.
        Using \eqref{equation:middlestep2}, we get
        \[ \frac{f_{(\mathbf{j}, \sigma)}}{f^\ast_{k,r}} \leq \frac{e^{s/r}}{\sqrt{2 \pi}^{s-1}} \sqrt{\frac{k-1}{\prod_{\ell = 1}^s j_i}}, \]
which together with \eqref{equation:prodji1} gives
        \[ \frac{f_{(\mathbf{j}, \sigma)}}{f^\ast_{k,r}} \leq \frac{e^{s/r} \sqrt{s}}{\sqrt{2 \pi}^{s-1}} < \frac{e \sqrt{s}}{\sqrt{2 \pi}^{s-1}}\le 1, \]
    provided  $s \geq 3$.  The case $s=1$ is trivial, so we can assume that $s = 2$.
    Now, using \eqref{equation:prodji2}, we get $j_1 j_2 \geq j_2(k-j_2-1) \geq k-2$, so by \eqref{equation:middlestep2} we have that
        \[  \frac{f_{(\mathbf{j}, \sigma)}}{f^\ast_{k,r}} \leq \frac{r}{r+1}   \frac{e^{2/r}}{\sqrt{2 \pi}} \sqrt{\frac{k-1}{\prod_{\ell = 1}^2 j_i}} \leq \frac{r}{r+1} \frac{e^{2/r}}{\sqrt{2 \pi}} \sqrt{\frac{k-1}{k-2}}. \]
        We have $r > s = 2$, so $r \geq 3$. Note that $re^{2/r}/(r+1) \leq 3 e^{2/3}/4$ for all $r \geq 3$.
        Also, $k - 1 = j_1 + j_2 \geq 2$, so $k \geq 3$.
        Using these two bounds, we get
        \[ \frac{f_{(\mathbf{j}, \sigma)}}{f^\ast_{k,r}} \leq \frac{r}{r+1} \frac{e^{2/r}}{\sqrt{2 \pi}} \sqrt{\frac{k-1}{k-2}} \leq \frac{3}{4}\frac{e^{2/3}}{\sqrt{2 \pi}} \sqrt{2}\le 1. \]

        \noindent \emph{Case 2: $s = r \geq 5$.}
        Now we may bound $W(\mathbf{j})$ as
        \begin{align}
            W(\mathbf{j})
            & \leq \frac{1}{r} \left( \sum_{\ell = 1}^r 
		\frac{rj_\ell + 1}{j_\ell + 1} \right)  +  \frac{1}{j_r + 1} \leq r + \frac{1}{j_r + 1} \leq r + \frac{1}{2},
        \label{equation:W}
		\end{align}
        where in the second inequality we used that $(rx+1)/(x+1) \leq r$ holds for all $x\ge 1$.
        Combining this with \eqref{equation:middlestep2} and \eqref{equation:prodji1}, and using that $r \geq 5$, we get
		\[ \frac{f_{(\mathbf{j}, \sigma)}}{f^\ast_{k,r}} \leq \frac{e \sqrt{r}}{\sqrt{2 \pi}^{r-1}} \left( r + \frac{1}{2} \right) < 1.\]
         \medskip
        
        \noindent \emph{Case 3: $s = r < 5$.} Since $r < 5$, by assumption we have $k \geq 23$ in this case.
        
        Suppose first that $s = r = 4$. 
        Using \eqref{equation:prodji2}, we get $j_1 j_2 j_3 j_4 \geq j_2(k-3-j_2) \geq k-4$.
        Together with \eqref{equation:middlestep2},   \eqref{equation:W}, and substituting $s=r=4$, we have
        \[ \frac{f_{(\mathbf{j}, \sigma)}}{f^\ast_{k,r}} \leq \frac{4}{5} \times \frac{e}{\sqrt{2 \pi}^{3}} \sqrt{\frac{k-1}{k-4}} \left( 4 + \frac{1}{2} \right) < 1,\]
        which holds for all $k \geq 6$.

        Now, assume $s = r = 3$.
        Note that if $j_2 \geq 2$, by \eqref{equation:prodji2} then we have that $j_1 j_2 j_3 \geq 2(k-4)$.
        Using this, together with \eqref{equation:middlestep2} and \eqref{equation:W}, we get
        \[ \frac{f_{(\mathbf{j}, \sigma)}}{f^\ast_{k,r}} \leq \frac{3}{4} \times \frac{e}{2\pi} \sqrt{\frac{k-1}{2(k-4)}} \left( 3 + \frac{1}{2} \right) < 1, \]
        where the last inequality uses $k \geq 10$.
        Hence, we can assume $s = r = 3$ and $j_2 = 1$. Therefore, $j_3 = 1$ and $j_1 = k-3$.
        This implies that $W(\mathbf{j}) = \frac{3k-8}{4k-8} + \frac{3}{2} \leq \frac{9}{4}$.
        Hence, from \eqref{equation:middlestep2} we get
        \[ \frac{f_{(\mathbf{j}, \sigma)}}{f^\ast_{k,r}} \leq \frac{3}{4} \times \frac{9}{4} \times \frac{e}{2\pi} \sqrt{\frac{k-1}{k-3}} < 1, \]
        where the last inequality holds for $k \geq 6$. 

        Finally, we can assume $s = r = 2$.
        If $j_2 \geq 4$, then by \eqref{equation:prodji2} we have $j_1 j_2 \geq 4(k-5)$.
        We also have $W(\mathbf{j}) \leq r+\frac{1}{2} = \frac{5}{2}$.
        Then, using \eqref{equation:middlestep2} we have
        \[ \frac{f_{(\mathbf{j}, \sigma)}}{f^\ast_{k,r}} \leq \frac{2}{3} \times \frac{e}{\sqrt{2 \pi}} \sqrt{\frac{k-1}{4(k-5)}} \times \frac{5}{2} < 1, \]
        where the last inequality holds for $k \geq 23$.

        Assume now that $s = r = 2$ and $j_2 = 3$; hence $j_1 = k-4$.
        We have $j_1 j_2 = 3(k-4)$, and we also have
        $W(\mathbf{j}) = \frac{3}{7} \left( \frac{2(k-4) + 1}{k-3} + \frac{7}{4} \right) + \frac{1}{4} = 1 + \frac{3}{7} \times \frac{2k - 7}{k-3} \leq \frac{13}{7}$.
        From \eqref{equation:middlestep2}, we deduce
        \[ \frac{f_{(\mathbf{j}, \sigma)}}{f^\ast_{k,r}} \leq \frac{2}{3} \times \frac{e}{\sqrt{2 \pi}} \sqrt{\frac{k-1}{3(k-4)}} \times \frac{13}{7} < 1, \]
        where the latter inequality holds for all $k \geq 9$.

        Assume now that $s = r = 2$ and $j_2 = 2$; hence $j_1 = k-3$.
        We have $j_1 j_2 = 2(k-3)$, and we also have
        $W(\mathbf{j}) = \frac{2}{5} \left( \frac{2(k-3)+1}{k-2} + \frac{5}{3} \right) + \frac{1}{3} = 1 + \frac{2}{5} \frac{2k-5}{k-2} \leq \frac{9}{5}$.
        From \eqref{equation:middlestep2}, we deduce
        \[ \frac{f_{(\mathbf{j}, \sigma)}}{f^\ast_{k,r}} \leq \frac{2}{3} \times \frac{e}{\sqrt{2 \pi}} \sqrt{\frac{k-1}{2(k-3)}} \times \frac{9}{5} < 1, \]
        where the latter inequality holds for all $k \geq 15$.

        Hence, we can assume that $s = r = 2$ and $j_2 = 1$.
        This implies that $j_1 = k-2$, so $\mathbf{j}=(k-2,1)$.
        Thus, from \eqref{equation:middlestep} we have
        \[\frac{f_{(\mathbf{j},1)}}{f^*_{k,2}}=\frac{3(2(k-2)+1)^{k-2}}{(2(k-1)+1)^{k-1}}\times \binom{k-1}{k-2}\times W(k-2,1).\]
A direct computation shows that 
\[W(k-2,1)=\frac{1}{3}\left(\frac{2(k-2)+1}{k-1}+\frac{3}{2}\right)+\frac{1}{2}=\frac{5k-6}{3k-3}.\]
Therefore, 
\begin{eqnarray*}\frac{f_{(\mathbf{j},1)}}{f^*_{k,2}}&=&\left(\frac{2k-3}{2k-1}\right)^{k-2}\cdot \frac{3k-3}{2k-1}\cdot \frac{5k-6}{3k-3}\\
&\le & \left(\frac{2k-3}{2k-1}\right)^{(2k-1)/2}\cdot \left(\frac{2k-3}{2k-1}\right)^{-3/2}\cdot \frac{5k-6}{2k-1} \\
&\le& e^{-1}\cdot \left(\frac{2k-3}{2k-1}\right)^{-3/2}\cdot \frac{5k-6}{2k-1}<1,
\end{eqnarray*}
where the last inequality holds for all $k \geq 16$.
This finishes the case $\sigma = 1$.

For $\sigma = -1$, the strategy is similar.
Note that we can assume $r \geq 3$, because $f_{((j_1, j_2), -1)} = f_{((j_1 - 1, j_2 - 1), 1)}$.
Let $r \geq 3$ and $\mathbf{j} = (j_1, \dotsc, j_r)$ satisfy $j_1 \geq j_2 \geq \dotsb \geq j_r \geq 1$ and $k+1 = \sum_{\ell = 1}^r j_\ell$.
Note this implies $k+1 \geq r$ as well.
By \cref{lemma:fkr-reduction2}, we have
	\begin{align*} f_{(\mathbf{j}, \sigma)}
		& = \binom{k+1}{j_1, \dotsc, j_r} \frac{\prod_{\ell = 1}^r (r j_\ell - 1)^{j_\ell}}{k^k r^{k-1}(k+1)} \tilde{W}(\mathbf{j}),
		\end{align*}
		where in the last step we wrote \[\tilde{W}(\mathbf{j}) :=  \frac{j_r}{r j_r - 1}\left( \left (\sum_{\ell = 1}^r \frac{j_\ell}{r j_\ell - 1} \right ) - \frac{1}{r j_r - 1} \right).\]
Using \cref{lemma:robbins} and similar bounds as before, we get
\begin{align}
     \frac{f_{(\mathbf{j}, \sigma)}}{f^\ast_{k,r}}
    & \label{extra2} = \binom{k+1}{j_1, \dotsc, j_r} \frac{1}{k(k+1)} \tilde{W}(\mathbf{j}) \frac{\prod_{\ell = 1}^r (r j_\ell - 1)^{j_\ell}}{(r(k-1)+1)^{k-1}} \\
    & \nonumber \leq \frac{1}{\sqrt{2 \pi}^{r-1}} \sqrt{\frac{k+1}{\prod_{\ell = 1}^r j_i}} \tilde{W}(\mathbf{j}) \frac{(k+1)^k}{k} \frac{r^{k+1}}{(r(k-1)+1)^{k-1}} \prod_{\ell=1}^r \left( 1 - \frac{1}{r j_\ell} \right)^{j_\ell} \\
    & \nonumber \leq \frac{1}{ e\sqrt{2 \pi}^{r-1}} \sqrt{\frac{k+1}{\prod_{\ell = 1}^r j_i}} \tilde{W}(\mathbf{j}) \frac{r^2(k+1)}{k} \left( \frac{r(k+1)}{r(k-1)+1} \right)^{k-1}\\
    & \leq \frac{e}{ \sqrt{2 \pi}^{r-1}} \sqrt{\frac{k+1}{\prod_{\ell = 1}^r j_i}} \tilde{W}(\mathbf{j}) \frac{r^2(k+1)}{k}. \label{extra}
    \end{align}

    Using $k+1 \geq r$ and $j_r \geq 1$, we get the simple bounds $\tilde{W}(\mathbf{j}) \leq 1/(r-1)$ and $(k+1)/k \leq r/(r-1)$. Plugging this in we get
    \begin{equation}\label{eq:sigma=-1}
\frac{f_{(\mathbf{j}, \sigma)}}{f^\ast_{k,r}} \leq \frac{e }{\sqrt{2 \pi}^{r-1}} \frac{r^3}{(r-1)^2} \sqrt{\frac{k+1}{\prod_{\ell = 1}^r j_i}}.
    \end{equation}
    Now we bound the term $\Omega := \sqrt{\frac{k+1}{\prod_{\ell = 1}^r j_i}}$ in different ways depending on whether $r\ge 6$ or not.
    We always have $\prod_{\ell = 1}^r j_i \geq j_1 \geq (k+1)/r$, which implies $\Omega \leq \sqrt{r}$.
    Plugging this into \eqref{eq:sigma=-1}, we are done for each $r \geq 6$.
    
    Suppose $r=5$, which implies that $k \geq 4$.
    If $j_2 \geq 2$, then $\prod_{\ell = 1}^5 j_i \geq 2 j_1 \geq 2(k+1)/5$, and therefore $\Omega \leq \sqrt{5/2}$, which again is enough in \eqref{eq:sigma=-1}.
    If $j_2 = 1$, then $\prod_{\ell = 1}^5 j_i = j_1 = k+2-r = k-3$.
    This implies that $\Omega = \sqrt{(k+1)/(k-3)}$. These values of $j_1, \dotsc, j_5$ yield that
    $\tilde{W}(\mathbf{j}) = \frac{19 k - 60}{80 k - 256}$. Combining these observations with
    \eqref{extra}, we are done for all $k \geq 5$; the only remaining case, $k=4$, can be checked directly.

    Suppose $r=4$; so $k \geq 23$.
    If $j_2 \geq 2$, we have $\prod_{\ell = 1}^4 j_i \geq 2(k-3)$, so $\Omega \leq \sqrt{\frac{k+1}{2(k-3)}}$.
    Using $\tilde{W}(\mathbf{j}) \leq 1/(r-1) = 1/3$ in \eqref{extra} we are done for all $k \geq 23$ (in fact $k \geq 8$ suffices).
    Hence we can assume $j_2 = 1$; so $\prod_{\ell = 1}^4 j_i = j_1 = k - 2$ and $\Omega = \sqrt{\frac{k+1}{k-2}}$.
    In this case, $\tilde{W}(\mathbf{j}) = \frac{11k - 24}{36k - 81}$. Again using these expressions in \eqref{extra} we are done for all $k \geq 23$ (in fact $k \geq 16$ suffices).

    Finally, suppose $r = 3$. If $j_3 \geq 2$, we have that 
    $\prod_{\ell = 1}^3 j_i \geq 4(k-3)$, so $\Omega \leq \sqrt{\frac{k+1}{4(k-3)}}$. Note that in this case 
     $\tilde{W}(\mathbf{j}) \leq 2/5$. Combining these observations with
    \eqref{extra}, we are done in this subcase for all $k \geq 23$.
    Next suppose $j_3 =1$ and $j_2 \geq 3$. We have that 
    $\prod_{\ell = 1}^3 j_i \geq 3(k-3)$, so $\Omega \leq \sqrt{\frac{k+1}{3(k-3)}}$. Note that in this case $j_3/(rj_3-1)=1/2$ and $j_2/(rj_2 -1)\leq 2/5$; further, as $k \geq 23$, we have that $j_1\geq 12$ and so 
     $j_1/(rj_1 -1)\leq 12/35$.
    This implies that 
     $\tilde{W}(\mathbf{j}) \leq 13/35$. Combining these observations with
    \eqref{extra}, we are done in this subcase for all $k \geq 23$. The final subcases when $j_3 =1$ and $j_2 \in [2]$ can be checked by applying \eqref{extra2} directly.
	\end{proof}

    We finish this section with a short derivation of Corollaries \ref{cor:vertex} and~\ref{corollary:fversusm}.
    To see \cref{cor:vertex} we just compare the values of $f_{k,2} = (1 - 1/(2k))^{k-1}$ given by \cref{lemma:maximiser} for $k \geq 5$ and the conjectured values $m'_{k} = 1 - ((k-1)/k)^{k-1} \leq m_{k}$ for the $1$-degree perfect matching threshold.
    The key point here is that $f_{k,2}$ is decreasing in $k$ and $m'_{k}$ is increasing in $k$, so the corollary follows by checking that $f_{16,2} > m'_{16}$ and $f_{17,2} < m'_{17}$.
    Similar monotonicity arguments yield \cref{corollary:fversusm}.

	\section{Conclusion}\label{sec:conc}

    In this paper, we have investigated the emergence of colour-bias in perfect matchings under minimum degree conditions.
    We conclude with a few remarks and open problems.

 \subsection{Stability for perfect matchings}\label{sec:stab}
In the cases where $f_{k,r} > m_{k}$, our proof of Theorem~\ref{thm:main} actually shows that if
$H$ has minimum vertex degree $\delta _1 (H) \geq (m_k+o(1))\binom{n-1}{k-1}$ 
then either $H$ has a perfect matching of significant colour-bias or $H$ is close (in edit distance) to an element of $\mathcal{F}^\ast_{k,r}$.
 
 
	\subsection{Cycles}

    A natural question in connection to our research is whether the results can be extended to cyclical structures such as tight Hamilton cycles.
    Formally, a $k$-uniform \emph{tight cycle} comes with a cyclical ordering of its vertices such that every $k$ consecutive vertices form an edge.
	A fundamental result of Rödl, Ruci\'nski and Szemerédi~\cite{RRS08a} states that an $n$-vertex $k$-graph with minimum codegree at least $(1/2+o(1))n$ contains a tight Hamilton cycle.
    Note that a tight Hamilton cycle of order divisible by $k$ contains a perfect matching,
    so the minimum degree thresholds for tight Hamiltonicity are bounded from below by the corresponding thresholds for perfect matchings.
    
    Colour-bias questions in this setting have been investigated by Mansilla Brito~\cite{Man23} as well as Gishboliner,  Glock and  Sgueglia~\cite{GGS23} for codegree conditions, who showed that in this situation the threshold for ordinary resp. colour-bias perfect matchings resp. tight Hamilton cycles coincide.
    For vertex degree however, a more nuanced picture emerges, which we illustrate with the following two problems.

    We believe that the  minimum vertex degree thresholds for colour-bias perfect matchings and tight Hamilton cycles agree in the case of $2$-edge-coloured $3$-graphs.
    Note that this behaviour differs from the ordinary threshold, which is known to be $5/9$~\cite{RRR19}.

    \begin{conjecture}\label{3conjtight}
        For every $\eps > 0$, there are $\gamma>0$ and $n_0\in \mathbb N$ such that every $2$-edge-coloured $3$-graph $H$ on $n \geq n_0$ vertices with $\delta_1(H) \geq \left(3/4+ \eps\right) \binom{n}{2}$ contains a tight Hamilton cycle with colour-bias above $\gamma n$.
    \end{conjecture}
Note that Conjecture~\ref{3conjtight} has the same extremal example~\cite[Example 4.1]{BTZ24} as the $(k,r) = (3,2)$ case of~\cref{lemma:maximiser-remaining}.

    It is plausible to suspect that this phenomenon persists for higher uniformities and the thresholds for colour-bias perfect matchings and tight Hamilton cycles are always the same.
    However, already for $4$-graphs the pattern changes.
	Indeed, consider the following construction.
	Partition the vertex set of $H$ into $V_1$ of size $7n/8$ and $V_2$ of size $n/8$ where $n$ is divisible by $8$.
	Add all edges that are not of type $(2,2)$.
	Colour the edges of type $(4,0)$ red and of type $(3,1)$ blue.
	The other edges can be coloured arbitrarily.
	For reasons of space and connectivity, a tight Hamilton cycle in $H$ can only be constructed using edges of type $(4,0)$ and $(3,1)$.
    Owing to the sizes of $V_1$ and $V_2$, a perfect matching formed by edges of type $(4,0)$ and $(3,1)$ has to be colour-balanced.
    Since $n$ is divisible by $4$, every tight Hamilton cycle in $H$ is the (edge-disjoint) union of $4$ perfect matchings and has thus to be colour-balanced too.
    On the other hand, the minimum vertex degree of $H$ is $\left({365}/{512} -o(1)\right) \binom{n}{3}$, which is larger than the threshold for colour-bias perfect matchings.
    We believe that this construction is sharp.
	
	\begin{conjecture}
         For every $\eps > 0$, there are $\gamma > 0$ and  $n_0\in \mathbb N$ such that every $2$-edge-coloured $4$-graph $H$ on $n \geq n_0$ vertices with $\delta_1(H) \geq \left(365/512+ \eps\right) \binom{n}{3}$ contains a tight Hamilton cycle with colour-bias above $\gamma n$.
	\end{conjecture}
	
	\subsection{Random hypergraphs}
	We can mimic the proof of our main result to find colour-bias perfect matchings in other situations; as long as for every $x, y \in V(H)$ we have at least $\Omega(n)$ vertex-disjoint $S \in N(x) \cap N(y)$.
	This property holds with high probability (w.h.p.) in a binomial random $k$-graph $H_k(n,p)$ if $p^2 n^{k-1} = \Omega(n)$, that is, if $p = \Omega(\sqrt{n^{2-k}})$.
    We sketch how to use the Key Lemma (\cref{lem:main}) in this setting.

    \begin{proposition} \label{proposition:random} 
        For any $k \geq 3$ and  $r \geq 2$, there exist $C_{k,r}, \gamma_{k,r} > 0$ such that, for $p \geq C_{k,r} n^{1 - k/2}$, and $n$ divisible by $k$, w.h.p. $H \sim H_k(n,p)$ has the property that, for each $r$-edge-colouring of $H$, there exists a perfect matching containing at least $n/(kr) + \gamma_{k,r}n$ edges of the same colour.
    \end{proposition}
For $k = 2$ and any $r$, an analogue of Proposition~\ref{proposition:random} is known to hold, even with the optimal discrepancy and (essentially) optimal value of $p$; see~\cite[Theorem 2]{GKM22}.
    \begin{proof}[Proof sketch]
        Let $C_{k,r}$ be sufficiently large, $p := C_{k,r} n^{1 - k/2}$, and $H := H_k(n,p)$.
        One can check that, w.h.p.,
        \begin{enumerate}[label={(\roman*)}]
            \item \label{item:random-indep} every set $X \subseteq V(H)$ of size at least $n/(2r)$ contains an edge of $H$, and
            \item \label{item:random-common} for each pair of distinct vertices $x, y \in V(H)$ and
            each set $X\subseteq V(H)$ of size at least $n/3$, we have that $N(x) \cap N(y)\cap X \neq \emptyset$.
        \end{enumerate}
        We also claim that, w.h.p.,
        \begin{enumerate}[label={(\roman*)}, resume]
            \item \label{item:random-matching} for every $X \subseteq V(H)$ of size at least $n/2$ and divisible by $k$, there exists a perfect matching in $H[X]$.
        \end{enumerate}
        To see this, we use a sprinkling argument.
        Let $p^\ast$ be such that $H_k(n/2, p^\ast)$ contains a perfect matching with probability at least $1/2$.
        By the work of Johannson, Kahn and Vu~\cite{JKV08} we know that we can take $p^\ast = O( n^{1 - k} \log n)$.
        Let $T := p / p^\ast$.
        Note that $T \geq \Omega(n^{k/2} / \log n) = \omega(n)$, since $k \geq 3$.
        For each $1 \leq i \leq T$, define $H_i: = H_k(n, p^\ast)$ on the same vertex set $V$.
        Then $H^\ast := \bigcup_{i=1}^{T} H_i$ distributes as $H_k(n, q)$ with $q \leq T p^\ast = p$, so it suffices to show that w.h.p. $H^\ast$ satisfies \ref{item:random-matching}.
        Indeed, for each set $X \subseteq V$ of size at least $n/2$, the probability of not containing a perfect matching in $H_i[X]$ is at most $1/2$, so for $H^\ast$ the probability is at most $2^{-T} = 2^{- \omega(n)}$, which allows us to take a union bound over the at most $2^n$ choices of $X$.

        Assuming that $H$ satisfies \ref{item:random-indep}--\ref{item:random-matching}, we can conclude the proof. Indeed,
        given any $r$-edge-colouring of $H$ and $Z \subseteq V(H)$ with $|Z| \leq n/2$, property \ref{item:random-indep} implies that $H \setminus Z$ does not belong to $\mathcal{F}_{k,r}$ (since $k$-graphs $G$ in $\mathcal{F}_{k,r}$ contain an independent set on a least $|V(G)|/r$ vertices and $H \setminus Z$ does not).
        If $|Z| \leq n/2$, then using property~\ref{item:random-common} repeatedly implies that there exist at least $k^2$ vertex-disjoint $(k-1)$-sets in $(N(x) \cap N(y))\setminus Z$, for every $x, y \in V(H)$.
        Thus, we can apply repeatedly the Key Lemma (\cref{lem:main}) to greedily find $n/(2k^4r)$ vertex-disjoint switchers in~$H$, each of size at most $k^2 + k$.
        After removing the switchers we are left with a set of vertices of size at least $n - (k^2 + k)n/(2k^4 r) \geq n/2$, so property~\ref{item:random-matching} implies that there exists a matching in $H$ covering precisely everything outside the switchers.
        This allows us to conclude as in the proof of \cref{thm:main}.
    \end{proof}
    We observe, however, that the threshold $\Theta(n^{-(k-1)}\log n)$ for getting a perfect matching w.h.p. in $H_k(n,p)$ is much lower than what Proposition~\ref{proposition:random} requires.
    We leave as an open problem  determining the correct threshold (depending on $k, r$) that always guarantees the existence of a colour-bias perfect matching in $H_k(n,p)$, no matter how we $r$-colour its edges.
    This question recently motivated Gishboliner, Glock, Michaeli and Sgueglia~\cite{GGMS25} to study matchings in sparsifications of edge-coloured hypergraphs, see~\cite[Theorem 1.4]{GGMS25} for more details.

    \subsection{Switchers and absorbers}
	
	In this paper, we showed that $k$-graphs whose minimum degree surpasses a certain threshold must contain a colour-bias perfect matching. Key to this result are certain small configurations called switchers, which are basically two superimposed matchings of different colour profiles.
	There is another well-known configuration in the context of perfect {matchings}, which is called an {absorber}. Formally, a $k$-uniform \emph{absorber} consists of two matchings such that one of them covers exactly $k$ more vertices than the other. Absorbers play a crucial role in the problem of finding perfect matchings under minimum degree conditions (see, e.g.,~\cite{HPS09}).
	In our proof, we exploited the fact that every pair of vertices shares many common neighbours. As it turns out, the same fact can be (and has been) used to construct absorbers for perfect matchings. It is thus natural to wonder whether there is a more direct connection between the two structures.
	We formalise this discussion as follows:
	
	The existence of absorbers can be captured with the concept of lattice completeness.	
	For a $k$-graph $H$ with vertex set $V$ and an edge $e \in E(H)$, we denote by $\vn_e \in \NATS _0^V$ the \emph{indicator vector}, which takes value $1$ at index $v$ if $v \in e$.
	We set $\vn_v = \vn_e$ for $e=(v)$.
	The \emph{lattice} of $H$ is the additive subgroup $\cL(H) \subset \INTS^{V}$ generated by the vectors $\vn_e$ with $e \in E(H)$.
	We say that $\cL(H)$ is \emph{complete} if it contains the \emph{transferral} $\vn_v - \vn_u$ for every $u,v \in V$.
	Note that the lattice of a $k$-graph is (trivially) complete if every two vertices of $H$ share a common neighbour.
	Moreover, it can be shown that a $k$-graph whose lattice is complete in a robust sense admits an absorber.
	On the other hand, it is not hard to see that a $k$-graph, which has a perfect matching after the deletion of any $k$ vertices, must have a complete lattice.
	In other words, lattice completeness is something we can (essentially) take as granted in the context of extremal perfect matching problems.
	It is therefore a tantalizing (if somewhat open-ended) question whether one can replace the common neighbourhood condition in \cref{lem:main} with a lattice completeness condition.

\section*{Acknowledgments}
Much of this research was undertaken at the 2nd Graph Theory in the Andes Workshop (March 25--29 2024). The authors are extremely grateful to Maya Stein for organising this event. 
We also wish to express our gratitude to Juan Perro for useful discussions and emotional support.

 We would like to thank the authors of \cite{luma} for interesting discussions after our respective papers were first written. In particular, we discussed the differences in the statements and proofs of our main results. A discussion on stability led us to add in our remark in Section~\ref{sec:stab}.
 We also provided them with a construction answering a question in their concluding remarks~\cite[Section 6]{luma}.

 Finally, we are grateful to the referees for their helpful and careful reviews.


\smallskip

{\noindent \bf Support statement.}

Hi\d{\^e}p H\`an was supported by the ANID Regular grant 1231599 and by ANID Basal Grant CMM FB210005, Richard Lang was supported by the EU Horizon 2020 programme MSCA (101018431), Jo\~ao Pedro Marciano was partially supported by FAPERJ (Proc.~E-26/200.977/2021) and CAPES, Mat\'ias Pavez-Sign\'e was supported by ANID-FONDECYT Regular grant 1241398 and ANID Basal Grant CMM FB210005, Nicol\'as Sanhueza-Matamala was supported by ANID-FONDECYT Iniciación Nº11220269 grant and ANID-FONDECYT Regular Nº1251121 grant, Andrew Treglown was supported by EPSRC grant EP/V048287/1, and Camila Z\'arate-Guer\'en was supported by EPSRC.

{\noindent \bf Open access statement.}
This research was funded in part by  EPSRC grant EP/V048287/1. For the purpose of open access, a CC BY public copyright licence is applied to any Author Accepted Manuscript arising from this submission.
 
{\noindent \bf Data availability statement.}
There are no additional data beyond that contained within the main manuscript.
    
	\bibliographystyle{plain}
	\bibliography{bibliography}

\begin{thebibliography}{10}

\bibitem{AFH+12}
N.~Alon, P.~Frankl, H.~Huang, V.~R\"{o}dl, A.~Ruci\'{n}ski, and B.~Sudakov.
\newblock Large matchings in uniform hypergraphs and the conjecture of {E}rd{\H{o}}s and {S}amuels.
\newblock {\em J. Combin. Theory Ser. A}, 119(6):1200--1215, 2012.

\bibitem{BCJ+20}
J.~Balogh, B.~Csaba, Y.~Jing, and A.~Pluh{\'a}r.
\newblock On the discrepancies of graphs.
\newblock {\em Electron. J. Combin.}, 27:P2.12, 2020.

\bibitem{BCP+21}
J.~Balogh, B.~Csaba, A.~Pluhár, and A.~Treglown.
\newblock A discrepancy version of the {H}ajnal--{S}zemer{\'e}di theorem.
\newblock {\em Comb. Probab. Comput.}, 30(3):444–459, 2021.

\bibitem{BTZ24}
J.~Balogh, A.~Treglown, and C.~Z{\'a}rate-Guer{\'e}n.
\newblock A note on colour-bias perfect matchings in hypergraphs.
\newblock {\em SIAM J. Discr. Math}, 38(4):1808--1839, 2024.

\bibitem{Bra22}
D.~Brada{\v{c}}.
\newblock Powers of {H}amilton cycles of high discrepancy are unavoidable.
\newblock {\em Electron. J. Combin.}, 29:P3.22, 2022.

\bibitem{BCG+23}
D.~Brada{\v{c}}, M.~Christoph, and L.~Gishboliner.
\newblock Minimum degree threshold for {$H$}-factors with high discrepancy.
\newblock {\em Electron. J. Combin.}, 31(3):P3.33, 2024.

\bibitem{Erd63}
P.~Erd\H{o}s.
\newblock Ramsey {\'e}s {V}an der {W}aerden t{\'e}tel{\'e}vel kapcsolatos kombinatorikai k{\'e}rd{\'e}sekr{\H{o}}l.
\newblock {\em Mat. Lapok}, 14:29--37, 1963.

\bibitem{ES72}
P.~Erd\H{o}s and J.~Spencer.
\newblock Imbalances in $k$-colorations.
\newblock {\em Networks}, 1:379--385, 1972.

\bibitem{FHL+21}
A.~Freschi, J.~Hyde, J.~Lada, and A.~Treglown.
\newblock A note on color-bias {H}amilton cycles in dense graphs.
\newblock {\em SIAM J. Discr. Math.}, 35(2):970--975, 2021.

\bibitem{GGMS25}
L.~Gishboliner, S.~Glock, P.~Michaeli, and A.~Sgueglia.
\newblock Defect and transference versions of the {A}lon--{F}rankl--{L}ov{\'a}sz theorem.
\newblock {\em arxiv:2503.05089}, 2025.

\bibitem{GGS23}
L.~Gishboliner, S.~Glock, and A.~Sgueglia.
\newblock Tight {H}amilton cycles with high discrepancy.
\newblock {\em Comb. Probab. Comput.}, to appear.

\bibitem{GKM22}
L.~Gishboliner, M.~Krivelevich, and P.~Michaeli.
\newblock Color-biased {H}amilton cycles in random graphs.
\newblock {\em Random Structures \& Algorithms}, 60(3):289--307, 2022.

\bibitem{GKP22}
L.~Gishboliner, M.~Krivelevich, and P.~Michaeli.
\newblock Discrepancies of spanning trees and {H}amilton cycles.
\newblock {\em J. Combin. Theory Ser. B}, 154:262--291, 2022.

\bibitem{HPS09}
H.~H{\`a}n, Y.~Person, and M~Schacht.
\newblock On perfect matchings in uniform hypergraphs with large minimum vertex degree.
\newblock {\em SIAM J. Discrete Math.}, 23(2):732--748, 2009.

\bibitem{JKV08}
A.~Johansson, J.~Kahn, and V.~Vu.
\newblock Factors in random graphs.
\newblock {\em Random Struct. Algorithms}, 33(1):1--28, 2008.

\bibitem{Kha16}
I.~Khan.
\newblock Perfect matchings in $4$-uniform hypergraphs.
\newblock {\em J. Combin. Theory Ser. B}, 116:333--366, 2016.

\bibitem{luma}
H.~Lu, J.~Ma, and S.~Xie.
\newblock Discrepancies of perfect matchings in hypergraphs.
\newblock {\em arXiv:2408.06020}, 2024.

\bibitem{Man23}
C.J. Mansilla~Brito.
\newblock Discrepancia de ciclos hamiltonianos en hipergrafos 3-uniformes.
\newblock Bachelor's thesis, Universidad de Concepci{\'o}n, 2023.

\bibitem{RRR19}
Chr. Reiher, V.~R{\"o}dl, A.~Ruci{\'n}ski, M.~Schacht, and E.~Szemer{\'e}di.
\newblock Minimum vertex degree condition for tight {H}amiltonian cycles in $3$-uniform hypergraphs.
\newblock {\em Proc. Lond. Math. Soc.}, 119(2):409--439, 2019.

\bibitem{Robbins1955}
H.~Robbins.
\newblock A remark on {S}tirling's formula.
\newblock {\em Amer. Math. Monthly}, 62:26--29, 1955.

\bibitem{RRS08a}
V.~R{\"o}dl, A.~Ruci{\'n}ski, and E.~Szemer{\'e}di.
\newblock An approximate {D}irac-type theorem for $k$-uniform hypergraphs.
\newblock {\em Combinatorica}, 28(2):229--260, 2008.

\bibitem{RRS09b}
V.~R{\"o}dl, A.~Ruci{\'n}ski, and E.~Szemer{\'e}di.
\newblock Perfect matchings in large uniform hypergraphs with large minimum collective degree.
\newblock {\em J. Combin. Theory Ser. A}, 116(3):613--636, 2009.

\bibitem{Zha16}
Y.~Zhao.
\newblock Recent advances on {D}irac-type problems for hypergraphs.
\newblock In {\em Recent {T}rends in {C}ombinatorics}, pages 145--165. Springer, [Cham], 2016.

\end{thebibliography}

\end{document}